\documentclass[reqno]{amsart}

\usepackage{amsthm}
\usepackage{amssymb}
\usepackage{mathrsfs}
\usepackage{graphicx}
\usepackage{mathtools}
\usepackage{todonotes}
\usepackage{enumitem}
\usepackage{xcolor}

{
\setcounter{enumi}{0}

\begin{enumerate}}%
{\end{enumerate} }

\makeatletter

\renewcommand*{\descriptionlabel}[1]{%
  \let\orglabel\label
  \let\label\@gobble
  \phantomsection
  \edef\@currentlabel{#1\unskip}%
  \let\label\orglabel
  \hspace\labelsep \upshape\bfseries #1%
}
\makeatother

{%
\setcounter{enumi}{0}

\begin{enumerate}}%
{\end{enumerate} }

\newcommand{\R}{\mathbb{R}}

	\newcommand{\Z}{\mathbb{Z}}

\def\N{\mathbb{N}}

\usepackage[colorlinks, citecolor=blue, linkcolor=blue, urlcolor=blue]{hyperref}

\newtheorem{thm}{Theorem}[section]
\newtheorem{lem}[thm]{Lemma}
\newtheorem{prop}[thm]{Proposition}
\newtheorem{cor}[thm]{Corollary}

\theoremstyle{definition}

\newtheorem{rem}[thm]{Remark}

\numberwithin{equation}{section}

\newcommand{\nc}{\normalcolor}

\allowdisplaybreaks
\DeclareRobustCommand{\SkipTocEntry}[5]{}

\title{Higher-order estimates of highest waves of the Whitham equation}
\author{Robin Østern Lien}
\thanks{Department of Mathematical Sciences,
Norwegian University of Science and Technology, H\o{}gskoleringen 1,
7034 Trondheim, Norway.\\
\hspace*{1.5em}Email: \texttt{robin.o.lien@ntnu.no}.}

\begin{document}
\maketitle

\makeatletter
\newcommand{\myitem}[1]{%
\item[#1]\protected@edef\@currentlabel{#1}%
}
\makeatother

\begin{abstract}
 The Whitham equation has given its name to a wider family of nonlocal, nonlinear equations with very weak dispersion, which all feature highest waves. Recent work by Ehrnström, Maehlen and Varholm establishes leading-order asymptotics at the crest of solutions to such Whitham-type equations, and conjectures similar expansions for all derivatives of the solution. By extending their techniques we confirm the conjecture for a range of equations with highest waves of Hölder regularity $C^{s}$ for  $s\in[0.35,1)$. We do this by strong induction, redistributing difference operators to deal with singularities in higher-order derivatives arising from integral convolution kernels.  The restriction $s\geq0.35$ arises solely from the separate argument establishing the zeroth-order asymptotics, which requires a uniform sign estimate that we verify using rigorous interval arithmetic. The higher-order induction itself applies for every $s\in(0,1)$, so any extension of the zeroth-order result immediately yields the corresponding higher-order expansions.
 \end{abstract}

\section{Introduction}
\noindent The Whitham Equation 
\begin{equation}
    \label{eq::whitham_equation}
    \partial_t \phi+ \partial_x(K * \phi + \phi^2)=0
\end{equation}
is a nonlocal and nonlinear dispersive shallow-water equation, modeling the surface profile $\phi=\phi(t,x)$ in time \(t\) and one-dimensional space \(x\). The kernel $K$ is defined via its Fourier transform,
\begin{equation}
\label{eq:integral_kernel_K}
    \hat{K}(\xi) = \int_{\mathbb{R}}K(x)e^{-ix\xi} \,\text{d}x =\sqrt{\frac{\text{tanh }\xi}{\xi}}.
\end{equation}
The equation \eqref{eq::whitham_equation} was introduced by Whitham \cite{Whitham1967VariationalMA},
who proposed replacing the approximate linear dispersion relation of the Korteweg--de Vries (KdV) equation by the exact linear dispersion relation of the gravity water-wave problem, while retaining its quadratic nonlinearity. Whitham conjectured that the interplay between the weak dispersion and the nonlinearity would give rise to a rich variety of phenomena, including wave breaking and highest waves. 

As with many other shallow water-wave models, the Whitham equation \eqref{eq::whitham_equation} admits both periodic and solitary traveling-wave solutions of the form $\phi(t,x)=\varphi(x-ct)$. Small-amplitude periodic traveling waves were established in \cite{ehrnstrom_kalisch_small_periodic_waves} through local bifurcation, while solitary solutions were established in \cite{Ehrnstrom_2012_existenceStabilitySolitaryWaves} by constrained minimization and concentration--compactness. The latter argument was further refined and generalized to less regular nonlinearities in \cite{Hildrum_2020}. Since then, solitary waves have also been obtained through Orlicz maximization \cite{atanas_orlicz_maximisation},
 implicit-function arguments \cite{wright_atanas_implicit_function}, and as limits of periodic waves \cite{katerina2023}. As conjectured by Whitham, the equation also features wave breaking \cite{HUR2017410, wavebreaking_revisited} and, importantly for us, highest waves --- both periodic \cite{Ehrnstrom_2019_highestCuspedWave} and solitary \cite{Truong_2021}. This sets \eqref{eq::whitham_equation} apart from more strongly dispersive equations, such as the KdV equation, since elliptic regularity rules out singularities when the dispersive operator is of positive order.

When the dispersive operator has negative order, however, singular highest waves become possible. It was conjectured already by Stokes that the limiting periodic gravity  wave develops a corner of interior angle \(2\pi/3\) at its crest \cite{stokes1847}, a conjecture proved by Amick, Fraenkel and Toland \cite{amick_stokes_conjecture} and Plotnikov \cite{plotnikov_stokes_conjecture}. Peaked travelling waves also occur in completely integrable models such as the Camassa--Holm \cite{camassa1993} and Degasperis--Procesi \cite{DegasperisHolmHone2002} equations, where they arise in explicit form through the integrable structure of the equations. Although explicit, these solutions do not directly reveal how the singularity depends on the underlying dispersion (but see \cite{ARNESEN201925}). In contrast, Ehrnström and Wahlén showed in \cite{Ehrnstrom_2019_highestCuspedWave} that, for Whitham-type equations, the regularity of the highest wave is determined by the singularity of the dispersive kernel, thereby making precise Whitham's original observation that the dispersion relation should determine the qualitative behavior of the equation \cite{Whitham1967VariationalMA}.\footnote{Analogous relationships between the order of an operator and the regularity of its solutions are well known in free-boundary and obstacle problems, but they arise there for different reasons~\cite{caffarelli2008}.} Subsequent work revealed that this relationship is remarkably stable. Namely, for Whitham-type equations with both homogeneous and inhomogeneous dispersive operators, the optimal regularity of the highest waves depends on the operator's order: they are of Hölder regularity \(C^s\) for operators of order ${-s}\in (-1,0)$ 
\cite{afram2021steady, orke2021}, log-Lipschitz for operators of order \(-s = -1\) \cite{ehrnstrom_johnson_2019, ehrnstrom_2023_onThePreciseCuspedBehaviour}, and Lipschitz for operators of order \(-s < - 1\) \cite{le2020waves, bruell2018waves}.

The singularity of the highest waves is therefore not always characterized by an angle. For the Whitham equation the \emph{precise} form of the wave is determined by identifying the constants \(c_1\) and \(c_2\) in
\[
c_1 |x|^{\frac{1}{2}} \leq \varphi(0)-\varphi(x)\leq c_2 |x|^{\frac{1}{2}} \qquad \text{as} \quad x \to 0.
\]
This was done by Ehrnström, Mæhlen and Varholm \cite{ehrnstrom_2023_onThePreciseCuspedBehaviour}, who determined the optimal constants to be \(c_1=c_2=\sqrt{\tfrac \pi8}\), after it was conjectured in \cite{Ehrnstrom_2019_highestCuspedWave}. They further establish the next-order limit
\begin{align*}\label{eq:u'_limit}
\lim_{x \rightarrow 0}\frac{-\varphi'(x)}{x^{-\frac12}} = \frac{1}{2}\sqrt{\frac{\pi}{8}},
\end{align*}
for the first derivative of a highest wave. Moreover, these results are obtained not only for the Whitham equation, but for two families of Whitham-type equations corresponding to dispersive orders $-s=-\frac{1}{2}$ and $-s=-1$. In particular, the limiting constants depend only on the order $-s$, and are independent of the period, the homogeneous or inhomogeneous nature of the dispersive operator, and the precise form of the nonlinearity within the class of equations considered.

This paper generalizes the homogeneous case of \cite{ehrnstrom_2023_onThePreciseCuspedBehaviour}.  We extend the induction to arbitrary derivatives for every  $s \in (0,1)$, which turns out to be subtler than anticipated since the asymptotic behavior $u^{(n)}\sim x^{s-n}$ gives rise to non-integrable singularities when differentiating for $n\geq2$\nc. We redistribute the $x$- and $h$-difference operators (as in \(\varphi(x \pm h)\)) before repeated differentiation, thereby obtaining modified difference identities, which preserve integrability and allow us to close the arguments. The zeroth-order asymptotics require a different argument from the higher-order induction, and we establish them for every  $s\in[0.35,1)$\nc. This proof relies on uniform-in-$s$ sign estimates, which are obtained using outward-rounding interval arithmetic  and subdivision of the relevant parameter domains. \nc

\medskip \noindent
\textbf{Structure of the paper.}
In Section~\ref{sec:setup} we introduce the class of homogeneous kernels $K$ and the profile $u$ (a translated version of the highest wave), and the assumptions on these objects. Then we state the Main Theorem \ref{thm::u^n_limit} and derive the modified central-difference identity \eqref{eq::u^n_central_diff_equation} for $u$ and its derivatives $u^{(m)}$, which forms the foundation of the subsequent analysis.
In Section~\ref{sec:diff_est} we use \eqref{eq::u^n_central_diff_equation} to estimate the central differences of $u^{(m)}$. These estimates form the technical core of the proof and are designed so that the dominated convergence theorem applies in Section~\ref{sec:comp_lim}. \nc 
In Section~\ref{sec:comp_lim} we prove the Main Theorem \ref{thm::u^n_limit} by carrying out the induction step. Specifically, we consider the quotient of \eqref{eq::u^n_central_diff_equation} by $2h$, then take $h\to0$ (for fixed $x$) and $x\to0$ to compute the asymptotic behavior of $u^{(m+1)}$ at the cusp. Throughout, limits as $x\to 0$ are taken from the right.
The proof for the asymptotic behavior of $u(x)$ as $x\to 0$ does not fit into the same framework and is therefore treated separately in Appendix~\ref{appendix_u_lim}.
 Finally, Appendix~\ref{appendix_hypergeom} contains the hypergeometric calculations needed in the final step of the proof of the Main Theorem \ref{thm::u^n_limit} in Section \ref{sec:comp_lim}.

\section{Setup and main result}\label{sec:setup}
We will work under the following assumptions:
\begin{enumerate}[itemsep=2pt]    
    \myitem{($K_s$)}\label{assump_K} Let $s\in(0,1)$ and $C>0$. The kernel $K$ is even, positive, and can be decomposed into a singular and regular part
    \begin{align*}
        K(x)=K_s(x)+K_{\text{reg}}(x), \qquad K_s(x)=C|x|^{s-1},
    \end{align*}     
    where $K_\text{reg}$ is real analytic on $\R$. Furthermore, $K$ and all of its derivatives are exponentially decaying.
    
    \myitem{($U1$)}\label{assump_U0} 
     $u\in C^\infty(\mathbb{R}\setminus P\mathbb{Z})$,
where $P\in(0,\infty]$ and $P\mathbb Z=\{0\}$ when $P=\infty$. Furthermore, $u$ is $P$-periodic when $P<\infty$, bounded, even, and nonnegative, and satisfies $u(0)=0$ and 
    \begin{equation}\label{eq:u_second_difference}
        u(x)^2 = \int_0^\infty  \left(K(y+x)-2K(y)+K(y-x)\right)u(y) dy.
    \end{equation}    
    
    \myitem{($U2$)}\label{assump_Ulim} Assume that $K$ is convex on $\mathbb{R}^+$ and that $u$ is increasing on $[0,\nu]$.
\end{enumerate}
\begin{rem}
    (a) The assumptions mirror those in \cite{ehrnstrom_2023_onThePreciseCuspedBehaviour} and are motivated by the Whitham equation, which satisfies these assumptions with $u(x)=\varphi(0)-\varphi(x)$, $s=\frac{1}{2}$ and $C=(2\pi)^{-\frac{1}{2}}$ \cite{Ehrnstrom_2019_highestCuspedWave, ehrnstrom_2023_onThePreciseCuspedBehaviour}. The assumptions can be weakened, but for the sake of readability we do not dwell on this issue. 

    \medskip
    \noindent (b) $P=\infty$ corresponds to solitary waves. The proofs require no modification in this case, since the arguments near the cusp are local and the estimates away from the cusp only rely on the boundedness of $u$ and the decay of $K$ and its derivatives.
    
    \medskip
    \noindent  (c) \ref{assump_Ulim} is stated separately since it is only  used for the $n=0$ proof of Main Theorem \ref{thm::u^n_limit} given in Appendix \ref{appendix_u_lim}. \nc
\end{rem}
\noindent We now present the Main Theorem:

\begin{thm}[Main Theorem]
\label{thm::u^n_limit}
    Let  $s\in[0.35,1)$, \nc and assume 
    \ref{assump_K}, \ref{assump_U0} and \ref{assump_Ulim}. Then the $n$-th derivative of $u$ admits the limit
    \begin{align*}
        u^{(n)}_*:=\lim_{x \rightarrow 0}\frac{u^{(n)}}{x^{s-n}} =\frac{C}{2}\frac{\Gamma(s)^2}{\Gamma(2s)}(-1)^n\frac{\Gamma(n-s)}{\Gamma(-s)},  \qquad n\in\N \cup \{0\},
    \end{align*}
    where $\Gamma$ denotes the gamma function.
\end{thm}
\noindent  The restriction $s\in[0.35,1)$ enters only through the zeroth-order asymptotics $u_*$,\footnote{ As they discuss in their paper, the approach of \cite{ehrnstrom_2023_onThePreciseCuspedBehaviour} for the zeroth-order asymptotics only works for $s\in(s_0,1)$ for some $s_0\approx \frac{1}{3}$, but fails for $s\in(0,s_0]$. \nc} proved separately in Appendix \ref{appendix_u_lim}. The difference estimate of Section \ref{sec:diff_est} and limit computation in Section \ref{sec:comp_lim} are valid for every $s\in(0,1)$. Consequently, an extension of the zeroth-order result to all $s\in(0,1)$ would immediately extend the Main Theorem to the full range $s\in(0,1)$ as well.\nc

As discussed above, the Whitham equation satisfies all the assumptions of the theorem, so we immediately get the following corollary:
\begin{cor}
    Let $u(x):=\varphi(0)-\varphi(x)$, where $\varphi$ denotes a highest cusped traveling-wave solution of the Whitham equation \eqref{eq::whitham_equation} from \cite{Ehrnstrom_2019_highestCuspedWave}. Then 
    \begin{align*}
        \lim_{x \rightarrow 0}\frac{u^{(n)}(x)}{x^{\frac{1}{2}-n}} = (-1)^{n+1}\frac{\Gamma(n-\frac{1}{2})}{2\sqrt{8}}, \qquad n\in\N \cup \{0\},
    \end{align*}
    where $u^{(n)}$ denotes the $n$-th derivative of $u$ and $\Gamma$ denotes the gamma function.
\end{cor}
\noindent We will prove Main Theorem \ref{thm::u^n_limit} through strong induction by adapting the framework from \cite{ehrnstrom_2023_onThePreciseCuspedBehaviour}, where they show the result for $n=0,1$ with $s=\frac{1}{2}$. To elucidate the inductive procedure we begin by presenting the main idea from \cite{ehrnstrom_2023_onThePreciseCuspedBehaviour} for establishing $u_*'$ from $u_*$, adapted to our setting. The starting point is the identity
 \begin{equation}\label{eq::u_central_difference_equation}\tag{$u$-diff}
        u(x+h)^2-u(x-h)^2 =-\int_0^{\infty}\delta_{2h}K(y)\delta_{2x}u(y)\text{ d}y,
\end{equation}
which follows from \eqref{eq:u_second_difference}, where $\delta_{2x}f := f(\cdot+x)-f(\cdot-x)$ denotes central differences. By the decomposition \ref{assump_K} we can write \eqref{eq::u_central_difference_equation} as
\begin{equation}
\label{u'''_limit_central_diff_eq_splitting}
\begin{aligned}
        u(x+h)^2-u(&x-h)^2 
        = -\bigg(\int_0^{\frac{2x}{3}}+
\int_{\frac{2x}{3}}^{2\nu}\bigg) \delta_{2h}K_{s}(y)\delta_{2x}u(y)\text{ d}y \\
&- \int_{0}^{2\nu}\delta_{2h}K_{\text{reg}}(y)\delta_{2x}u(y) dy - \int_{2\nu}^{\infty}\delta_{2h}K(y)\delta_{2x}u(y) dy,
\end{aligned}
\end{equation}
when $0<2h<x<\nu$, for a fixed, small enough $\nu$ isolating the origin. 
Note that 
\begin{align*}
    \lim_{x \rightarrow 0} \lim_{h\rightarrow 0} x^{1-2s}\frac{u(x+h)^2-u(x-h)^2}{2h} = \lim_{x \rightarrow 0} 2\frac{u(x)}{x^s} \frac{u'(x)}{x^{s-1}} = 2u_* u_*'.
\end{align*}
Computing the same limits for the right-hand side of \eqref{u'''_limit_central_diff_eq_splitting} then lets us solve for $u'_*$, but to rigorously justify this one needs difference estimates on $|\delta_{2h}u(x)|$ so the dominated convergence theorem applies. 

The expected asymptotic behavior $u^{(n)} \sim x^{s-n}$ near the origin introduces new technical challenges since $x^{s-n}$ has a non-integrable singularity at the origin for $n\geq 2$. These were not present in \cite{ehrnstrom_2023_onThePreciseCuspedBehaviour} where only $u_*$ was needed to establish $u_*'$. At infinity, however, $n\geq 2$ is helpful since the faster decay of $x^{s-n}$ makes the tail easier to control. Consequently, $u_*'$, which will be the base case in our proof, must be treated separately. \nc

The result for $n=0$ in Main Theorem \ref{thm::u^n_limit} does not follow from the same framework as for $n\geq 1$ and must be handled completely differently. We include the proof for completeness, but as the $n=0$ case is not our main focus we relegate it to Appendix \ref{appendix_u_lim}. The approach is essentially identical to \cite{ehrnstrom_2023_onThePreciseCuspedBehaviour}, however in their case the final part of the proof is greatly simplified as they consider the specific case $s=\frac{1}{2}$.  As the dependence on $s$ is highly involved in the expressions that appear, getting precise enough uniform bounds in $s$ on the range  $[0.35,1)$ \nc  is difficult.  We therefore use a computer-assisted approach by partitioning the relevant parameter domains into boxes and rigorously checking the sign on each box using outward-rounded interval arithmetic. \nc

\medskip
\noindent \textbf{The modified difference identity. } \nc Our first course of action is to obtain a difference identity analogous to 
\eqref{u'''_limit_central_diff_eq_splitting} for higher derivatives.  This identity will allow us to derive estimates on $\delta_{2h}u^{(n)}$
 and subsequently compute the asymptotic behavior of $u^{(n+1)}$. The natural idea is to differentiate \eqref{u'''_limit_central_diff_eq_splitting} $n$ times. To make this possible, we first redistribute the $h$- and $x$-difference operators so that each resulting integral contains only singular kernels that remain integrable after differentiation. Near the origin we keep the difference operators on $u$ to avoid non-integrable kernels, while away from the origin we place all differences on $K$, allowing us to exploit the boundedness of $u$. After suitable changes of variables and using that $u$ and $K$ are even functions, we obtain from \eqref{u'''_limit_central_diff_eq_splitting} the following identity when $0<2h<x<\nu$: 

\begin{equation}
\begin{aligned}
    &u(x+h)^2-u(x-h)^2 = -\int_{0}^{\frac{2x}{3}}\delta_{2h}K_{s}(y)\delta_{2y}u(x) dy \\
    &+ \int_{\frac{5x}{3}}^{2\nu+x}u(y)\big(\delta_{2h}K_{s}(x+y)+\delta_{2h}K_{s}(x-y)\big)dy +\int_{-\frac x3}^{\frac{5x}{3}}u(y)\delta_{2h}K_{s}(x+y)dy \\
    &+ \int_{-2\nu-x}^{2\nu+x}u(y)\delta_{2h}K_{\text{reg}}(x-y) dy + \int_{2\nu+x}^{\infty}u(y)\big(\delta_{2h}K(x+y)+\delta_{2h}K(x-y)\big) dy.
\end{aligned}\label{eq::u'''_limit_central_diff_eq_full_splitting}
\end{equation}
 Each integral on the right-hand side is absolutely integrable by \ref{assump_K} and \ref{assump_U0}. 
 
 Now we can differentiate \eqref{eq::u'''_limit_central_diff_eq_full_splitting} $m$ times to retrieve the $u^{(m)}$-difference identity. We begin by differentiating the left-hand side. By induction and Pascal's rule, 
\begin{align*}
    \frac{d^m}{dx^m}u(x\pm h)^2  = \sum_{i=0}^{m}{m \choose i}u^{(i)}(x\pm h)u^{(m-i)}(x\pm h).
    \end{align*}
Then we write
\begin{align*}
    \frac{d^m}{dx^m}\delta_{2h}\left(u(x)^2 \right)  
    &= \left(u(x+h)+u(x-h)\right)\delta_{2h} u^{(m)}(x) \\ 
    &\qquad +\frac{1}{2}\sum_{i=1}^m { m \choose i } \left(u^{(i)}(x+h)+u^{(i)}(x-h)\right)\delta_{2h}u^{(m-i)}(x) \\
    & \qquad + \frac{1}{2}\sum_{i=0}^{m-1} { m \choose i }\left(u^{(m-i)}(x+h)+u^{(m-i)}(x-h)\right) \delta_{2h} u^{(i)}(x),
\end{align*}
since we want to isolate the $\delta_{2h}u^{(m)}$-terms. Now we differentiate the right-hand side of \eqref{eq::u'''_limit_central_diff_eq_full_splitting} (this is justified below). Then we rearrange the differentiated equation, and multiply by $x^{1-2s}(x-h)^m$ which is needed to control the singularities at $x=h$ and $x=0$. This yields the $u^{(m)}$-difference identity for $m\geq 1$, $0<2h<x<\nu$:

\begin{equation}\label{eq::u^n_central_diff_equation}\tag{$u^{(m)}$-diff}
    \begin{aligned}
        D_0 :&= x^{1-2s}(x-h)^m\left(u(x+h)+u(x-h)\right)\delta_{2h}u^{(m)}(x) \\
        &= x^{1-2s}(x-h)^m \bigg[ -\int_{0}^{\frac{2x}{3}}\delta_{2h}K_{s}(y)\delta_{2y}u^{(m)}(x) dy \\
        &\hspace{4mm}  + \int_{\frac{5x}{3}}^{2\nu+x}u(y)\big(\delta_{2h}K^{(m)}_{s}(x+y)+\delta_{2h}K^{(m)}_{s}(x-y)\big) dy \\
        &\hspace{4mm}+ \int_{-\frac x3}^{\frac{5x}{3}}u(y)\delta_{2h}K^{(m)}_{s}(x+y)dy \\
        &\hspace{4mm}+\int_{-2\nu-x}^{2\nu+x}u(y)\delta_{2h}K^{(m)}_{\textnormal{reg}}(x-y)dy \\ 
        &\hspace{4mm}+ \int_{2\nu+x}^{\infty}u(y)\big(\delta_{2h}K^{(m)}(x+y)+\delta_{2h}K^{(m)}(x-y)\big)dy \\
        &\hspace{4mm} +\sum_{i=0}^{m-1} \delta_{2h}K^{(i)}_{s}\big(\tfrac{2x}{3}\big)\Big((-1)^{i}u^{(m-1-i)}\big(\tfrac{5x}{3}\big)+u^{(m-1-i)}\big(\tfrac{x}{3} \big) \Big)\\
        &\hspace{4mm} - \frac{ 1}{2}\sum_{i=1}^m { m \choose i } \left(u^{(i)}(x+h)+u^{(i)}(x-h)\right)\delta_{2h}u^{(m-i)}(x) \\
    & \hspace{4mm} - \frac{ 1}{2}\sum_{i=0}^{m-1} { m \choose i }\left(u^{(m-i)}(x+h)+u^{(m-i)}(x-h)\right) \delta_{2h}u^{(i)}(x) \bigg] \\
        &=:I_1+I_2+I_3+I_4+I_5+I_b + D_1 + D_2.
    \end{aligned}
\end{equation}
The difference terms $D_1,D_2$ were moved over from the left-hand side, while all boundary terms from differentiating the integrals are collected in $I_b$. The differentiation of the integrals is justified by considering forward difference quotients of each integral. Using the standard and integral version of the mean value theorem, which both hold by the smoothness of $u$ and $K$ away from their singularities, each difference can be rewritten as boundary terms and an integral whose integrand is bounded by an integrable function by the assumptions above. Hence the dominated convergence theorem applies.

\medskip
\section{Establishing the difference estimate}\label{sec:diff_est}
In this section we use \eqref{eq::u^n_central_diff_equation} to prove bounds on the central differences $\delta_{2h}u^{(m)}$ near the cusp (the origin). These bounds are necessary for the dominated convergence theorem to apply in the limit computation in Section~\ref{sec:comp_lim}.
\begin{lem} 
\label{lem::u^(n)_limit_first_estimate}
    Let  $n\in \N\cup\{0\}$, $s\in(0,1)$, and assume \ref{assump_K}, \ref{assump_U0}. Let $\nu>0$ such that $3\nu < \min\{1, P/2\}$, and $0<h<x<\nu$. Assume also 
    \begin{enumerate}[itemsep=2pt]
    \myitem{(A1)}\label{assump_1}
        There are constants $0<c_1\leq c_2$ such that $c_1x^{s-i} \leq |u^{(i)}(x)| \leq c_2x^{s-i}$ for all $i\in \{0,1,\dots,n \}$ uniformly in $(0,3\nu)$. 
    \end{enumerate}
Then the estimate  
        \begin{align}\label{eq::u^n_first_estimate}
            |x^{1-s}(x-h)^{m}(u^{(m)}(x+h)-u^{(m)}(x-h))|\lesssim h
        \end{align}\nc 
        holds uniformly on $(0, 2\nu )$ for all $m\in \{0,1,\dots,n\}$.
\end{lem}
\begin{rem}\label{rem:a1_implies}
Note that \ref{assump_1} holds for some $\nu$ if the limits $\lim_{x\rightarrow 0} x^{i-s}u^{(i)}(x)$ exist and are nonzero. Furthermore, by the evenness of $u$ we have that $|u^{(i)}|$ is even for all $i$, and consequently the assumption holds on $(-3\nu, 3\nu)$ by replacing $x$ with $|x|$. Furthermore, this and the $P$-periodicity of $u$ implies that the same bound holds around all the cusps: $c_1 |jP-x|^{s-i}\leq |u^{(i)}(x)|\leq c_2|jP-x|^{s-i}$ holds for all $i \in \{0,1,\dots, n \}$ uniformly in $(jP-3\nu, jP+3\nu)$ for all $j\in \Z$.    
\end{rem}
\begin{proof}
     Consider first $h \in [\tfrac x2,x)$. For all $m\in\{0,\dots,n\}$ we have by \ref{assump_1} that 
    \begin{align*}
        &|(x-h)^m(u^{(m)}(x+h)-u^{(m)}(x-h))| \\
        & \lesssim (x-h)^m((x+h)^{s-m}+(x-h)^{s-m}) \lesssim (\tfrac{x-h}{x+h})^m(x+h)^{s} +(x-h)^s\lesssim h^{s},
    \end{align*}
    since $x\pm h\leq 2h\pm h$ and $\frac{x-h}{x+h}\leq 1$. We then apply this to find that
    \begin{align*}
        &|x^{1-s}(x-h)^m(u^{(m)}(x+h)-u^{(m)}(x-h))|\lesssim x^{1-s} h^s \leq (2h)^{1-s}h^s \simeq h.
    \end{align*}\nc 
    
    We now proceed by strong induction to prove the claim for $h < \frac{x}{2}$. To this end we work with \eqref{eq::u^n_central_diff_equation}, which holds for $0<2h< x < \nu $, and we will estimate each term in the equation separately. Specifically, we will show that $D_0$ is estimated from below by $(x-h)^n x^{1-s}|\delta_{2h}u^{(n)}(x)|$  and that the terms on the right-hand side admit $O(h)$ bounds. The difficult term is $I_1$ which does not seem to admit the required $O(h)$ bound by direct estimation. To circumvent this we use a bootstrap argument inspired by the proof of \cite[Lemma 5.6]{ehrnstrom_2023_onThePreciseCuspedBehaviour}\footnote{ There is a small mistake in the computation preceding \cite[Equation (5.21)]{ehrnstrom_2023_onThePreciseCuspedBehaviour}: after the change of variables $y=\tau h$, the denominator should be $h^{1/2}$, not $x^{1/2}$. The erroneous factor gives the direct $O(h)$ estimate for their analogue of our $I_1$-term, but with the correct scaling this is precisely the term that needs to be handled by the bootstrap argument. }, adapting its more general form communicated by M.~Ehrnstr\"om. 
    
    The base case $n=0$ must be treated separately, and we save it for last to reuse as many computations as possible. In all following estimates we explicitly use that $0<2h<x<\nu$.
    
    \medskip\noindent \textbf{Inductive step $n\geq 1$.} Assume that \eqref{eq::u^n_first_estimate} holds for $m\in\{0, \dots, n-1\}$. We must show that it then also holds for $m=n$. 
    
    $\bullet$ \textbf{$D_0$:} For the left-hand side of \eqref{eq::u^n_central_diff_equation}, \ref{assump_1} and the nonnegativity of $u$  yield 
    \begin{align*}
        |D_0|  
        &\gtrsim x^{1-2s}(x-h)^n((x+h)^{s}+(x-h)^{s})|\delta_{2h}u^{(n)}(x)|  \geq (x-h)^nx^{1-s}|\delta_{2h}u^{(n)}(x)|.
    \end{align*}
    Now we estimate the right-hand side of \eqref{eq::u^n_central_diff_equation}. Note that $h\leq x^{1-s}h^s$.
    
    $\bullet$ \textbf{$D_1, D_2$:} By the triangle inequality and \ref{assump_1},
    \begin{align*}
        |D_1|
        &\lesssim \sum_{i=1}^n x^{s-i} x^{i} x^{1-2s} (x-h)^{n-i}|\delta_{2h}u^{(n-i)}(x)| \lesssim\sum_{i=1}^n h \simeq h.
    \end{align*}
    In the second inequality we use the inductive hypothesis. $|D_2|\lesssim h$ is analogous.

$\bullet$ \textbf{$I_b$:} 
     Note that 
     \begin{align}\label{eq:K_s_deriv}
         K_s^{(i)}(x)=C(-\text{sign}(x))^i\frac{\Gamma(i+1-s)}{\Gamma(1-s)}|x|^{s-(i+1)}. 
     \end{align}  
     By the triangle inequality, \eqref{eq:K_s_deriv}, and \ref{assump_1},
    \begin{align*}
        |I_b|& \lesssim \sum_{i=0}^{n-1} \left| \left(\tfrac{2}{3}x+h\right)^{s-(i+1)}-\left(\tfrac{2}{3}x-h\right)^{s-(i+1)}\right|x^{s-(n-1-i)}(x-h)^n x^{1-2s} \\
        & \leq \sum_{i=0}^{n-1} \left| \left(\tfrac{2}{3}+\tfrac hx\right)^{s-(i+1)}-\left(\tfrac{2}{3}-\tfrac hx \right)^{s-(i+1)} \right| x \leq  \sum_{i=0}^{n-1} 2\cdot6^{i+2-s}(i+1-s) \frac{h}{x}  x \lesssim h,
    \end{align*}
    where we use the mean value theorem and that $\frac{h}{x}\leq \frac{1}{2}$.
        
    $\bullet$ \textbf{$I_4, I_5$:} The final two integrals are easy to handle by \ref{assump_K} and since $u$ is bounded:
    \begin{align*}
        &|I_4| \lesssim h \nu^{n+{1-2s}} \|u\|_\infty \int_{-2\nu-x}^{2\nu+x}\frac{|\delta_{2h}K^{(n)}_{\text{reg}}(x-y)|}{2h}dy 
        \lesssim h\nu^{n+2-2s}\|u\|_\infty \|K_{\text{reg}}^{(n+1)}\|_\infty \simeq h,\\[2mm]
        & |I_5| 
         \lesssim \nu^{n+{1-2s}} \|u\|_\infty \int_{2\nu+x}^{\infty}\int_{-h}^h 2|K^{(n+1)}(x-y+t)| \, dt \, dy \\
         &\qquad \simeq \int_{-h}^h \int_{2\nu+x}^{\infty} 
         \big| K^{(n+1)}(x-y+t)\big| \, dy \, dt \simeq \int_{-h}^h  \, dt \simeq 
         h.
    \end{align*}
    In the last line we use that $\big|K^{(n+1)}(x-y+t)\big|$ is integrable on $(2\nu,\infty)$ since $K$ has exponentially decaying derivatives and $x-y+t$ is bounded away from zero. We also used that $n\geq 1 \implies (n+1-2s)>0$.

    $\bullet$ \textbf{$I_2, I_3$:} \ref{assump_1}, \eqref{eq:K_s_deriv} and the mean value theorem yield, for $\xi_y, \widetilde{\xi}_y\in(-1,1)$,
    \begin{align*}
         |I_2|& \lesssim x^{n+{1-2s}} \int_{\frac{5x}{3}}^{2\nu+x}y^{s}\Big( \frac{h}{(x+y-\xi_y h)^{n+2-s}}+ \frac{h}{|x-y-\widetilde{\xi}_y h|^{n+2-s}}\Big)\text{ d}y \\
        & \leq  x^{n+{1-2s}} h \int_{\frac{5x}{3}}^{2\nu +x} \frac{y^{s}}{(y-\frac{9}{10}y)^{n+2-s}} \text{ d}y  \simeq x^{n+{1-2s}} h \int_{\frac53}^{\frac{2\nu}{x} +1} \frac{x\text{ d}\tau}{(x\tau)^{n+2-2s}}  \lesssim  h,
    \end{align*}
    since the final integral converges on $(\tfrac53,\infty)$. We have used that $x+h \leq \frac{3}{2}x\leq \frac{9}{10}y$. By similar arguments, 
    \begin{align*}
        |I_3|&\lesssim 
         x^{n+{1-2s}} h \int_{-\frac{1}{3}}^{\frac{5}{3}} \frac{|x\tau|^s}{x^{n+2-s}(\tau+1-\frac{h}{x})^{n+2-s}} x \, d\tau \leq h \int_{-\frac{1}{3}}^{\frac{5}{3}}\frac{|\tau|^s}{(\tau+\frac{1}{2})^{n+2-s}}d\tau\simeq h,
    \end{align*} 
    since the final integral converges. Here we have used that by \ref{assump_1}, $u(y)\lesssim |y|^s$ on $(-\tfrac{x}{3}, \tfrac{5x}{3})$ since $u$ is even (see Remark \ref{rem:a1_implies}).

   $\bullet$ \textbf{Bootstrapping.} Combining what we have so far yields
   \begin{align} \label{eq:bootstrapping_object}
       (x-h)^nx^{1-s}|\delta_{2h}u^{(n)}(x)| \lesssim h+|I_1|.
   \end{align}
   Following the bootstrap and interpolation argument from \cite[Lemma 5.6]{ehrnstrom_2023_onThePreciseCuspedBehaviour}, define $$C_m(\alpha) := \sup_{0<h<x<\nu} \frac{(x-h)^m |\delta_{2h}u^{(m)}(x)|}{h^{\alpha}x^{s-\alpha}},$$ 
   where we emphasize that the supremum is in $h$ and $x$, but not $\nu$. Observe that $C_m(0) < \infty$ by \ref{assump_1}.  Furthermore, fix $\lambda \in (1-s,1)$ such that $s\frac{1-\lambda^k}{1-\lambda}\neq1$ for any nonnegative integer $k$. Such a $\lambda$ exists since the left-hand side is strictly increasing in $\lambda$ for $k\geq 2$ (and does not equal $1$ for $k=0,1$), so the set of excluded values is at most countable. Define $$\alpha_{k+1}:=s+\lambda \alpha_k \qquad \text{with} \qquad \alpha_0=0,$$
    and note by the formula for a geometric series that 
    \begin{align}\label{eq:bootstrap_alpha_limit}
        \alpha_{k} = s\frac{1-\lambda^k}{1-\lambda} \xrightarrow{k\to \infty} \frac{s}{1-\lambda}>1,
    \end{align}
    since $\lambda >1-s$. Let $k^*$ denote the first integer such that $\alpha_{k^*}>1$.

    For $y\in(0,\frac{2}{3}x)$, the definition of $C_m(\alpha)$ and  \ref{assump_1} yield respectively that
    \begin{align*}
        (x-y)^n|\delta_{2y}u^{(n)}(x)| \lesssim C_n(\alpha)y^\alpha x^{s-\alpha} \qquad \text{and} \qquad         (x-y)^n|\delta_{2y}u^{(n)}(x)| \lesssim x^s.
    \end{align*}
    We interpolate between these estimates using the well-known inequality $\min(a,b)\leq a^cb^{1-c}$ for all $a,b\geq0$ and $c\in[0,1]$:
    \begin{align*}
        (x-y)^n|\delta_{2y}u^{(n)}(x)| \lesssim C_n(\alpha)^\lambda y^{\alpha \lambda} x^{s-\alpha \lambda}.
    \end{align*}
    From this, the fact that $\frac{(x-h)^n}{(x-y)^n}\lesssim 1$, and \eqref{eq:bootstrapping_object} we get 
    \begin{align}\label{eq:bootstrapping_object_2}
         (x-h)^nx^{1-s}|\delta_{2h}u^{(n)}(x)| \lesssim h+C_n(\alpha)^\lambda x^{1-s-\lambda\alpha}\int_0^\frac{2x}{3}|\delta_{2h}K_s(y)|y^{\alpha \lambda} dy. 
    \end{align}
    Consider the integral. Note that 
    \begin{align}\label{eq:delta_K_s_identity}
        \delta_{2h}K_s(\tau h) = \frac{C}{|\tau h +h|^{1-s}}-\frac{C}{|\tau h - h|^{1-s}} 
        = C h^{s-1}\delta_2(|\tau|^{s-1}).
    \end{align}
    Changing variables $y=\tau h$ and splitting the integral, we therefore find that
    \begin{equation}\label{eq:I_1_int_auxilliary_computation}
    \begin{aligned}
        &\int_0^\frac{2x}{3}|\delta_{2h}K_s(y)|y^{q} dy \lesssim h^{s+q}\big(\int_0^\frac{4}{3} |\delta_2(|\tau|^{s-1})| \tau^q d\tau+\int_\frac{4}{3}^\frac{2x}{3h} \tau^{s-2+q} d\tau \big) \\
        &\lesssim h^{s+q}+\begin{cases}
            \frac{h^{s+q}}{1-(s+q)}, \qquad & q<1-s, \\
            h\ln(\frac{x}{h}), &q=1-s, \\
            \frac{hx^{s+q-1}}{s+q-1}, & q>1-s.
        \end{cases} =:J_q(h,x)
    \end{aligned}
        \end{equation}
    The integral on $(0,\frac{4}{3})$ converges since $|\tau\pm 1|^{s-1}$ is integrable on $(0,\tfrac{4}{3})$, while the integral on $(\frac{4}{3},\frac{2x}{3h})$ follows since  $|\delta_2(|\tau|^{s-1})|\leq 2^{5-2s}(1-s)\tau^{s-2}$ on $(\frac{4}{3},\infty)$ by the mean value theorem.

    Consider now $k<k^*$ and $\alpha=\alpha_k$ in \eqref{eq:bootstrapping_object_2}. Then \eqref{eq:bootstrapping_object_2} and \eqref{eq:I_1_int_auxilliary_computation} yield
    \begin{align*}
        (x-h)^nx^{1-s}|\delta_{2h}u^{(n)}(x)| &\lesssim h+C_n(\alpha_k)^\lambda x^{1-s-\lambda\alpha_k}J_{\lambda \alpha_k}(h,x).
    \end{align*}
    If $k+1<k^*$ we find by \eqref{eq:I_1_int_auxilliary_computation} (then, $\alpha_{k+1}=s+\lambda\alpha_k<1$ by the definition of $\lambda$ so we never encounter the logarithmic case) that
    \begin{align*}
        (x-h)^nx^{1-s}|\delta_{2h}u^{(n)}(x)| &\lesssim 
        h+C_n(\alpha_k)^\lambda x^{1-\alpha_{k+1}}h^{\alpha_{k+1}}\big(1+\tfrac{1}{1-\alpha_{k+1}})
    \end{align*}
    Dividing by $x^{1-\alpha_{k+1}}h^{\alpha_{k+1}}$ and using that $h=h^{\alpha_{k+1}}h^{1-\alpha_{k+1}}\lesssim h^{\alpha_{k+1}}x^{1-\alpha_{k+1}}$ yields
    \begin{align}\label{eq:C_m_recursion}
        C_n(\alpha_{k+1}) \lesssim 1 + C_n(\alpha_k)^\lambda\big(1+\tfrac{1}{1-\alpha_{k+1}}\big).
    \end{align}
    If $k+1=k^*$ we have $q+s=\alpha_{k^*}>1$ in  \eqref{eq:I_1_int_auxilliary_computation} and use $h^{\alpha_{k+1}}\lesssim hx^{\alpha_{k+1}-1}$ to get  
    \begin{equation}\label{eq:bootstrapping_final_step}
    \begin{aligned}
        (x-h)^nx^{1-s}|\delta_{2h}u^{(n)}(x)| &\lesssim 
                h+C_n(\alpha_k)^\lambda \big(x^{1-\alpha_{k^*}}h^{\alpha_{k^*}}+\tfrac{h}{\alpha_{k^*}-1}\big) \\
        & \lesssim h \Big(1+C_n(\alpha_k)^\lambda\big(1+\tfrac{1}{\alpha_{k^*}-1}\big)\Big) \lesssim h.
    \end{aligned}
        \end{equation}
Recalling that $C_n(0)<\infty$, repeated application of \eqref{eq:C_m_recursion} yields $C_n(\alpha_{k})<\infty$ for all $k=0,\dots,k^*-1$, whence \eqref{eq:bootstrapping_final_step} yields the claimed estimate. \nc

    \medskip\noindent \textbf{Base case $n=0$. } In this case the terms $D_1, D_2, I_b$ are absent. Observe that the above estimate for $I_3$ remains valid and  that the bootstrap argument applies unchanged with $C_0(\alpha)$\nc. All that remains is therefore to modify the $I_2,I_4,I_5$ estimates. We rewrite the integrals by applying the mean-value theorem twice, using that $K'$, $K'_{\text{reg}}$, and $K'_s$ are odd: 
    \begin{equation}\label{eq:integrals_base_case_formulation}
    \begin{aligned}
          I_2 
        &= x^{1-2s}2h\int_{\frac{5x}{3}}^{2\nu+x} u(y) K''_s(y+\eta_y)(2x+\xi_y+\overline{\xi}_y) dy, \\
        I_4 &= x^{1-2s}2h\int_{0}^{2\nu +x} u(y) K''_{\text{reg}}(y+\eta_y)(2x+\xi_y+\overline{\xi}_y) dy, \\
        I_5 &= x^{1-2s}2h\int_{2\nu +x}^\infty u(y) K''(y+\eta_y)(2x+\xi_y+\overline{\xi}_y) dy,   
    \end{aligned}
        \end{equation}
    where $\xi_y, \overline{\xi}_y\in(-h,h)$, $\eta_y\in(-x-\overline{\xi}_y, x+\xi_y)$ are different in each line. For $I_4$ we also split $\int_{-2\nu-x}^{2\nu+x}=\int_{-2\nu-x}^{0}+\int_{0}^{2\nu+x}$ and changed variables $y \mapsto -y$ in the first integral and used that $u$ is even. Consequently,
    \begin{equation}\label{eq:integral_bounds_base_case}
    \begin{aligned}
        |I_2|&\lesssim x^{2-2s}h\int_{\frac{5x}{3}}^{2\nu +x} y^s|K_s''(y-\tfrac{3x}{2})| dy \lesssim x^{2-2s}h\int_{\frac{5}{3}}^\infty \frac{(x\tau)^sx d\tau}{(\tau x-\tfrac{3x}{2})^{3-s}} \simeq h \\
        |I_4|&\lesssim x^{2-2s}h \|u\|_\infty \| K''_{\text{reg}}\|_\infty 3\nu \lesssim\nu^{2-2s}h, \\ 
        |I_5| &\lesssim x^{2-2s}h\|u \|_\infty\int_{2\nu}^\infty \sup_{|y-r|<\frac{3\nu}{2}} |K''(r)| dy \lesssim \nu^{2-2s}h.
    \end{aligned}
        \end{equation}
    For $I_2$ we use \ref{assump_1}, $\tau-\tfrac{3}{2}\geq \tfrac{\tau}{10}$, and that the integral converges since $3-2s>1$.
\end{proof}

\section{Computing the limit}\label{sec:comp_lim} Now we arrive at the proof of Main Theorem \ref{thm::u^n_limit}. We return to the central difference identity \eqref{eq::u^n_central_diff_equation}, but now divide by $2h$, and let $h\to0$ for fixed $x$. Then we let $x\to0$ to extract the leading-order asymptotics of $u^{(m+1)}$ at the cusp. The difference estimate from Section~\ref{sec:diff_est} is used throughout to justify the dominated convergence theorem. We also need the following Lemma:
\begin{lem}\label{lem:combinatorics_identity}
    Let $s\in (0,1)$. The following identity holds for $n\in \N$, with the convention $\frac{1}{\Gamma(2s-n)}\vert_{s=\frac{1}{2}}=0$:
    \begin{align*}
        \sum_{i=1}^n{n +1 \choose i}\Gamma(n-i+1-s)\Gamma(i-s) = 
        \left(\frac{\Gamma(2s)\Gamma(s-n)}{\Gamma(s)\Gamma(2s-n)}-1\right)2\Gamma(n+1-s)\Gamma(-s).
    \end{align*}
\end{lem}
\begin{proof} This follows from induction, Pascal's rule, and $\tfrac{\Gamma(z)}{z-1}=\Gamma(z-1)$.
\end{proof}
\noindent Tedious calculations requiring the use of hypergeometric functions are necessary in the proof of Main Theorem \ref{thm::u^n_limit}; these have been relegated to Appendix \ref{appendix_hypergeom}. Note that $m=k$ in \eqref{eq::u^n_central_diff_equation} will give us $u_*^{(k+1)}$, so the base case $n=0$ in the proof is the proof of $u_*'$, not $u_*$.

\begin{proof}[Proof of Main Theorem \ref{thm::u^n_limit}] The proof for the existence and value of $u_*$ is given in Appendix \ref{appendix_u_lim}, and will be used below. We first consider $n\geq 1$ and show the base case $n=0$ afterwards. 

\medskip
\noindent\textbf{Inductive step $n\geq 1$. } Assume that the claim holds for $m \in \{0,\dots, n \}$, $n\geq1$. This implies that \ref{assump_1} is true for $m\in \{0,\dots, n\}$, and consequently Lemma \ref{lem::u^(n)_limit_first_estimate} holds for $m\in \{0,\dots, n\}$. We now prove that the claim also holds for $m=n+1$. We will use \eqref{eq::u^n_central_diff_equation}. To ease the notation, let in the following 
    $$D:=\frac{C}{2}\frac{\Gamma(s)^2}{\Gamma(2s)}.$$
    
    $\bullet$ \textbf{$D_0$:} For the left-hand side of \eqref{eq::u^n_central_diff_equation}, $u_*$ (recall, this limit is known) yields 
    \begin{align*}
        \frac{D_0}{2h} 
        &\xrightarrow[]{h\rightarrow0}2\frac{u(x)}{x^s}\frac{u^{(n+1)}(x)}{{x^{s-(n+1)}}} \xrightarrow[]{x\rightarrow 0} 2 u_* \lim_{x\rightarrow 0}\frac{u^{(n+1)}(x)}{{x^{s-(n+1)}}}=2D u_*^{(n+1)}. 
    \end{align*}

    $\bullet$ \textbf{$D_1, D_2$:} By the inductive hypothesis,
    \begin{align*}
        \frac{D_1}{2h}
        &\xrightarrow[]{h \rightarrow 0}-\sum_{i=1}^n {n \choose i}\frac{u^{(i)}(x)}{x^{s-i}}\frac{ u^{(n-i+1)}(x)}{x^{s-(n-i+1)}} \xrightarrow[]{x\rightarrow 0}-\sum_{i=1}^n {n \choose i}u^{(i)}_*u^{(n-i+1)}_*.
    \end{align*}
    Similarly, but also by changing variables $i \mapsto i+1$ and Pascal's rule, 
\begin{align*}
    \lim_{x\rightarrow 0} \lim_{h\rightarrow 0}\frac{D_2}{2h} 
    &=-\sum_{i=1}^{n} {n+1 \choose i}u^{(n-i+1)}_*u^{(i)}_* +\sum_{i=1}^{n} {n \choose i}u^{(n-i+1)}_*u^{(i)}_*.
\end{align*}
Thus we have that 
\begin{align*}
    &\lim_{x \rightarrow 0} \lim_{h \rightarrow 0} \left(\frac{D_1}{2h}+\frac{D_2}{2h} \right) 
     = \frac{(-1)^{n}D^2}{\Gamma(-s)^2}\sum_{i=1}^{n}{n+1 \choose i} \Gamma(n-i+1-s)\Gamma(i-s).
\end{align*}

    $\bullet$ \textbf{$I_b$:} By the inductive hypothesis and \eqref{eq:K_s_deriv}, 
    \begin{align*}
        &\frac{I_b}{2h} 
                \xrightarrow[]{h\rightarrow 0} \sum_{i=0}^{n-1} \frac{K^{(i+1)}_{\textnormal{s}}\big(\tfrac{2x}{3}\big)}{x^{s-(i+2)}}\left((-1)^{i}\frac{u^{(n-1-i)} \big(\tfrac{5x}{3}\big)}{\left(\frac{3}{5}\cdot \frac{5x}{3}\right)^{s-(n-1-i)}}+\frac{u^{(n-1-i)} \big(\tfrac{x}{3}\big)}{\left(3\cdot \frac{x}{3}\right)^{s-(n-1-i)}} \right) \\
        &\xrightarrow[]{x \rightarrow 0}
        C\sum_{i=0}^{n-1} (-1)^{i+1}\frac{3^{n+1-2s}}{2^{i+2-s}}\frac{\Gamma(i+2-s)}{\Gamma(1-s)}\left((-1)^{i}5^{s-(n-1-i)}+ 1 \right)u^{(n-1-i)}_*.
    \end{align*}
    
    $\bullet$ \textbf{$I_4, I_5$:} Since $n\geq 1$ we have that
    \begin{align*}
        \frac{I_4}{2h} &=x^{1-2s}(x-h)^n\int_{-2\nu-x}^{2\nu+x} u(y) \frac{\delta_{2h}K^{(n)}_{\text{reg}}(x-y)}{2h} dy \\
        &\xrightarrow[]{h\rightarrow 0}x^{n+1-2s}\int_{-2\nu-x}^{2\nu+x} u(y)K^{(n+1)}_{\text{reg}}(x-y) dy \xrightarrow[]{x\rightarrow 0}0
    \end{align*}
    by the dominated convergence theorem, as the mean-value theorem yields
    \begin{align*}
         \left|u(y) \frac{\delta_{2h}K^{(n)}_{\text{reg}}(x-y)}{2h}\right| \leq \|u \|_\infty \| K_{\text{reg}}^{(n+1)}\|_\infty.
    \end{align*}
    This is independent of $h$ and $x$ and integrable for $s\in(0,1)$ on $(-3\nu, 3\nu)$. We also have 
    \begin{align*}
        \frac{I_5}{2h}&=x^{1-2s}(x-h)^n\int_{2\nu+x}^{\infty} u(y) \left( \frac{\delta_{2h}K^{(n)}(x+y)}{2h}+\frac{\delta_{2h}K^{(n)}(x-y)}{2h} \right) dy \\
        &\xrightarrow[]{h\rightarrow 0}x^{n+1-2s}\int_{2\nu+x}^{\infty} u(y) \left( K^{(n+1)}(x+y)+K^{(n+1)}(x-y) \right) dy \xrightarrow[]{x\rightarrow 0}0
    \end{align*}
    by the dominated convergence theorem, since the mean-value theorem yields
    \begin{align*}
        & \left| u(y) \left( \frac{\delta_{2h}K^{(n)}(x+y)}{2h}+\frac{\delta_{2h}K^{(n)}(x-y)}{2h} \right) \right| 
         \leq \|u\|_\infty2 \sup_{|y-r|<\frac{3\nu}{2}}\left|K^{(n+1)}(r)\right|.
    \end{align*}
    This is independent of $h$ and $x$ and integrable on $(2\nu, \infty)$ by the exponential decay.

    $\bullet$ \textbf{$I_2, I_3$:} By $u_*$ and \eqref{eq:K_s_deriv},
    \begin{align*}
        &\frac{I_2}{2h} = 
        x^{1-2s}(x-h)^n\int_{\frac{5}{3}}^{\frac{2\nu}{x}+1}u(\tau x)\left( \frac{\delta_{2h}K^{(n)}_{\textnormal{s}}(x+\tau x)}{2h}+\frac{\delta_{2h}K^{(n)}_{\textnormal{s}}(x-\tau x)}{2h}\right)  x  d\tau \\
        &\xrightarrow[]{h\rightarrow 0} 
        C\tfrac{\Gamma(n+2-s)}{\Gamma(1-s)} \int_{\frac{5}{3}}^{\frac{2\nu}{x}+1}\frac{ u(\tau x)}{(\tau x)^s} \tau^s\left((1+\tau)^{s-(n+2)}(-1)^{n+1} +(\tau-1)^{s-(n+2)} \right)  d\tau \\
        &\xrightarrow[]{x \rightarrow 0} C\tfrac{\Gamma(n+2-s)}{\Gamma(1-s)} u_* \int_{\frac{5}{3}}^{\infty} \tau^s\left((-1)^{n+1} (1+\tau)^{s-(n+2)}+(\tau-1)^{s-(n+2)} \right)  d\tau.
    \end{align*}
    Here the application of the dominated convergence theorem holds by the following.  By the mean-value theorem, $u_*$, \eqref{eq:K_s_deriv}, and since $(\tau-1-\frac{h}{x})\geq \frac{\tau}{10}$,
    \begin{align*}
          &\left| x^{1-2s}(x-h)^n u(\tau x)\left( \frac{\delta_{2h}K^{(n)}_{\textnormal{s}}(x+\tau x)}{2h}+\frac{\delta_{2h}K^{(n)}_{\textnormal{s}}(x-\tau x)}{2h}\right)x \right| \\
                    &\leq x^{2-s+n}\frac{u(\tau x)}{(\tau x)^s} \tau^s 2|K_s^{(n+1)}(\tau x- x-h)|  \lesssim x^{2-s+n} \frac{\tau^s }{(\tau x-x-h)^{n+2-s}} \\ 
                    &=\frac{\tau^s}{(\tau-1-\frac{h}{x})^{n+2-s}}\lesssim\frac{\tau^s}{\tau^{n+2-s}}= \frac{1}{\tau^{n+2-2s}},
        \end{align*}
    which is independent of $h$ and $x$ and integrable on $(\frac{5}{3},\infty)$ for $s\in(0,1)$ since $n\geq 1$. 
    
Similarly, \eqref{eq:K_s_deriv} and $u_*$ yield
 \begin{align*}
        \frac{I_3}{2h} &= x^{1-2s}(x-h)^n\int_{-\frac{1}{3}}^{\frac{5}{3}} u(\tau x) \frac{\delta_{2h}K^{(n)}_{\textnormal{s}}(x+\tau x)}{2h} x \textnormal{ d}\tau \\
        &\xrightarrow[]{h\rightarrow 0} 
        C\frac{\Gamma(n+2-s)}{\Gamma(1-s)}\int_{-\frac{1}{3}}^{\frac{5}{3}}\frac{ u(\tau x)}{|\tau x|^s} |\tau|^s (1+\tau)^{s-(n+2)}(-1)^{n+1}  \textnormal{ d}\tau \\
        &\xrightarrow[]{x \rightarrow 0} C(-1)^{n+1}\frac{\Gamma(n+2-s)}{\Gamma(1-s)} u_* \int_{-\frac{1}{3}}^{\frac{5}{3}}|\tau|^s  (1+\tau)^{s-(n+2)}  \textnormal{ d}\tau,
    \end{align*}
    since, by the mean-value theorem, \eqref{eq:K_s_deriv} and \ref{assump_1}, 
    \begin{align*}
        &\left| x^{1-2s}(x-h)^n u(\tau x)  \frac{\delta_{2h}K^{(n)}_{\textnormal{s}}(x+\tau x)}{2h} x\right|   \leq   \left| x^{1-2s+n+1} u(\tau x) K^{(n+1)}_{\textnormal{s}}(\tau x+x-h) \right| \\
    &\qquad \lesssim \left| \frac{u(\tau x)}{|\tau x|^s}|\tau|^s (\tau +\tfrac12)^{s-(n+2)} \right|  \lesssim |\tau|^s(\tau+\tfrac12)^{s-(n+2)}.
    \end{align*}
    This is independent of $h$ and $x$ and integrable on $(-\frac{1}{3}, \frac{5}{3})$.

     $\bullet$ \textbf{$I_1$:} 
    To use the dominated convergence theorem we must split the integral as
    \begin{align*}
        \frac{I_1}{2h}&=-x^{1-2s}(x-h)^n \left( \int_{0}^{\frac{4h}{3}}+\int_{\frac{4h}{3}}^{\frac{2x}{3}} \right) \frac{\delta_{2h}K_{s}(y)}{2h}\delta_{2y}u^{(n)}(x)\textnormal{ d}y =: \frac{I_A}{2h}+\frac{I_B}{2h}.
    \end{align*}
     This is because $K_s(y-h)$ is not continuous at $y=h$, which prevents us from using the mean value theorem directly on $\delta_{2h}K_s(y)$ on $(0,\frac{2x}{3})$. By \eqref{eq:delta_K_s_identity} we write 
\begin{align*}
     \frac{I_A}{2h}= -x^{1-2s}(x-h)^n\int_0^{\frac{4}{3}} \tfrac{C}{2} h^{s-1} \delta_2(|\tau|^{s-1}) \delta_{2(\tau h)}u^{(n)}(x) d\tau.
\end{align*}
 By Lemma \ref{lem::u^(n)_limit_first_estimate}, 
\begin{align*}
    &\left|-x^{1-2s}(x-h)^n \tfrac{C}{2} h^{s-1} \delta_2(|\tau|^{s-1}) \delta_{2(\tau h)}u^{(n)}(x) \right|\lesssim  \frac{x^{1-2s+n} h^{s-1}|\delta_2(|\tau|^{s-1})|\tau h}{(x-\tau h)^n x^{1-s}} \\
    &\leq \frac{x^nh^{s}}{x^{s}(x-\frac{2}{3}x)^n}|\delta_2(|\tau|^{s-1})|\tau \simeq \frac{h^{s}}{x^s}|\delta_2(|\tau|^{s-1})|\tau \lesssim |\delta_2(|\tau|^{s-1})|\tau,
\end{align*}
which is independent of $h$ and $x$ and integrable on $(0,\frac{4}{3})$ for $s\in(0,1)$. Thus the dominated convergence theorem and Lemma \ref{lem::u^(n)_limit_first_estimate} yield
\begin{align*}
\left| \frac{I_A}{2h} \right| \lesssim \frac{h^{s}}{x^s}\int_0^{\frac{4}{3}}|\delta_2(|\tau|^{s-1})|\tau d\tau \xrightarrow[]{h\rightarrow0}0 \quad \implies \quad \lim_{x\rightarrow 0} \lim_{h\rightarrow 0}\frac{I_A}{2h}=\lim_{x\rightarrow 0}0=0.
\end{align*}

For the second integral, by \eqref{eq:K_s_deriv} and the inductive hypothesis for $m=n$,
\begin{align*}
    &\frac{I_B}{2h}= -x^{1-2s}(x-h)^n \int_{\frac{4h}{3x}}^{\frac{2}{3}}  \frac{\delta_{2h}K_{s}(\tau x)}{2h}\delta_{2(\tau x)}u^{(n)}(x) x \textnormal{ d}\tau \\
    & \xrightarrow[]{h\rightarrow 0} 
    C\tfrac{\Gamma(2-s)}{\Gamma(1-s)} \int_{0}^{\frac{2}{3}}  \tau^{s-2} \left((1+\tau)^{s-n} \frac{u^{(n)}(x+\tau x)}{(x+\tau x)^{s-n}}-(1-\tau)^{s-n} \frac{u^{(n)}(x-\tau x)}{(x-\tau x)^{s-n}} \right)  \textnormal{ d}\tau \\
    &\xrightarrow[]{x \rightarrow 0} C(1-s) u^{(n)}_*\int_{0}^{\frac{2}{3}}  \tau^{s-2} \left((1+\tau)^{s-n}-(1-\tau)^{s-n} \right)  \textnormal{ d}\tau,
\end{align*}
since by the mean-value theorem, Lemma \ref{lem::u^(n)_limit_first_estimate}, \eqref{eq:K_s_deriv}, and that $(\tau-\frac{h}{x})\geq \frac{\tau}{4}$, 
\begin{align*}
    &\left|-x^{1-2s}(x-h)^n \frac{\delta_{2h}K_{s}(\tau x)}{2h}\delta_{2(\tau x)}u^{(n)}(x) x \right| \lesssim  \frac{x^{2-2s}(x-h)^n |K'_s(\tau x -h)|\tau x}{(x-\tau x)^nx^{1-s}} \\
    &\leq \frac{x^{2-s }x^n}{(x-\frac{2}{3}x)^n}\frac{\tau}{(\tau x-h)^{2-s}} \simeq \frac{\tau}{(\tau - \tfrac{h}{x})^{2-s}} \lesssim \frac{\tau}{\tau^{2-s}} \simeq \tau^{s-1}.
\end{align*}
This is independent of $h$ and $x$ and is integrable on $(0,\frac{2}{3})$ for $s\in(0,1)$, so the dominated convergence theorem applies. \nc

$\bullet$ \textbf{Combining the limits.} Combining the above limits of the integral terms yields an expression consisting of explicit integrals and finite sums:
\begin{equation*}
    \begin{aligned}
         I^n_*:&=\lim_{x\rightarrow 0}\lim_{h \rightarrow 0} \frac{I_1+I_2+I_3+I_4+I_5+I_b}{2h} \\
    &= C(1-s) u^{(n)}_*\int_{0}^{\frac{2}{3}}  \tau^{s-2} \left((1+\tau)^{s-n}-(1-\tau)^{s-n} \right)  \textnormal{ d}\tau \\
    &\quad + C\frac{\Gamma(n+2-s)}{\Gamma(1-s)} u_* \int_{\frac{5}{3}}^{\infty} \tau^s\left((-1)^{n+1} (1+\tau)^{s-(n+2)}+(\tau-1)^{s-(n+2)} \right) d\tau \\
    &\quad + C(-1)^{n+1}\frac{\Gamma(n+2-s)}{\Gamma(1-s)} u_* \int_{-\frac{1}{3}}^{\frac{5}{3}}|\tau|^s  (1+\tau)^{s-(n+2)}  \textnormal{ d}\tau \\
    &\quad+C\sum_{i=0}^{n-1} (-1)^{i+1}\frac{3^{n+1-2s}}{2^{i+2-s}}\frac{\Gamma(i+2-s)}{\Gamma(1-s)}\left((-1)^{i}5^{s-(n-1-i)}+ 1 \right)u^{(n-1-i)}_*.
    \end{aligned}
\end{equation*}
After integrating by parts so each integral contains a factor $\tau^{s-1}$ (which is indeed allowed as all objects involved are convergent), using the inductive hypothesis and absorbing the boundary terms into the sum, this becomes
\begin{equation}\label{eq:big_computation}
\begin{aligned}
    I^n_*&=(-1)^{n+1}CD\frac{\Gamma(n+1-s)}{\Gamma(-s)}\Bigg[\int_{0}^{\frac{2}{3}}\frac{\tau^{s-1}}{(1-\tau)^{n+1-s}} d\tau+\int_{0}^{\frac{1}{3}}\frac{\tau^{s-1}}{(1-\tau)^{n+1-s}} d\tau  \\
    &\quad -3^{n+1-2s}\sum_{i=0}^{n-1} \frac{1}{2^{i+1-s}}\frac{\Gamma(i+1-s)\Gamma(n-i-s)}{\Gamma(1-s)\Gamma(n+1-s)}\\
    &\quad -\int_{\frac{2}{3}}^{\infty}\tau^{s-1}(1+\tau)^{s-1-n} d\tau+(-1)^n\int_{\frac{5}{3}}^{\infty}\tau^{s-1}(\tau-1)^{s-1-n} d\tau  \\
    &\quad -3^{n+1-2s}\sum_{i=0}^{n-1}(-1)^{i+1} \frac{5^{s-(n-i)}}{2^{i+1-s}}\frac{\Gamma(i+1-s)\Gamma(n-i-s)}{\Gamma(1-s)\Gamma(n+1-s)}\Bigg].
\end{aligned}
\end{equation}
This expression can be simplified further. By changing variables $\tau \mapsto (\tau-1)$ and integrating by parts $n$ times we find that 
\begin{align*}
    \int_{\frac{2}{3}}^\infty \frac{-\tau^{s-1} d\tau}{(1+\tau)^{n+1-s}} =\int_{\frac{5}{3}}^{\infty}\frac{(-1)^{n+1} \tau^{s-1} d\tau}{(\tau-1)^{n+1-s}}  +\sum_{i=0}^{n-1} (-1)^{i+1}\frac{(\frac 53)^{s-n+i}}{(\frac 23)^{i+1-s}} \tfrac{\Gamma(i+1-s)\Gamma(n-i-s)}{\Gamma(1-s)\Gamma(n+1-s)}.
\end{align*}
Thus the two final lines of \eqref{eq:big_computation} cancel (we have also used Euler's reflection formula). This, together with Proposition \ref{appendix_prop_1} and the difference term limits then yield
\begin{align*}
&u_*^{(n+1)}= \frac{1}{2D}\lim_{x\rightarrow 0}\lim_{h \rightarrow 0} \frac{I_1+I_2+I_3+I_4+I_5+I_b+D_1+D_2}{2h} \\
     &=\frac{D(-1)^{n+1}\Gamma(n+1-s)}{\Gamma(-s)}\bigg(\frac{\Gamma(2s)}{\Gamma(s)} \frac{\Gamma(s-n)}{\Gamma(2s-n)} -\frac{\sum_{i=1}^{n}{n+1 \choose i} \Gamma(n-i+1-s)\Gamma(i-s)}{2\Gamma(n+1-s)\Gamma(-s)}\bigg)\\
   &=D(-1)^{n+1}\frac{\Gamma(n+1-s)}{\Gamma(-s)}.
\end{align*}
The final equality holds by Lemma \ref{lem:combinatorics_identity}.

\medskip
\noindent \textbf{Base case $n=0$. } In this case we do not have  $D_1,D_2, I_b$, and the $D_0$, $I_1$ and $I_3$ estimates above still work for $n=0$ without modification. $I_2$ yields the same limit as above, but for the dominated convergence bound we must now use that $K_s'$ is odd and the mean-value theorem twice:
\begin{align*}
    &\left| x^{2-2s} u(\tau x)\left( \frac{\delta_{2h}K_{s}(x+\tau x)}{2h}+\frac{\delta_{2h}K_{s}(x-\tau x)}{2h}\right) \right| \\
                    &\lesssim \tfrac{u(\tau x)}{(\tau x)^s} \tfrac{\tau^s}{x^{s-2}}|K_s'(\tau x+ x+\xi_y)-K_s'(\tau x-x-\overline\xi_y)| \lesssim \tfrac{\tau^s}{x^{s-3}}|K''_s(\tau x-\tfrac{3}{2}x)| 
                    \lesssim \frac{1}{\tau^{3-2s}},
\end{align*}
by the existence of $u_*$ and since $(\tau-\frac{3}{2}) \geq \frac{1}{10}\tau$, which is integrable on $(\frac{5}{3},\infty)$ for $s\in(0,1)$. Finally, for $I_4$ and $I_5$ we use the formulations in \eqref{eq:integrals_base_case_formulation}. We find that
\begin{align*}
    \frac{I_5}{2h}&\xrightarrow{h\rightarrow 0} 2x^{2-2s}\int_{2\nu +x}^\infty u(y)K''(y+\eta_y) dy \xrightarrow{x\rightarrow 0} 0,\\ 
    \frac{I_4}{2h}&\xrightarrow{h \rightarrow 0} 2x^{2-2s}\int_0^{2\nu +x}u(y)K''_{\text{reg}}(y+\eta_y)dy \xrightarrow{x\rightarrow 0} 0,
\end{align*}
by the dominated convergence theorem, since  computations similar to \eqref{eq:integral_bounds_base_case} yield integrable bounds independent of $h$ and $x$. Combining the limits and changing variables similarly to in the inductive step above yields
\begin{align*}
    &\lim_{x\to 0}\lim_{h\to 0}\frac{I_1+I_2+I_3+I_4+I_5}{2h}
        =CDs\bigg( \Big(\int_0^{\frac{2}{3}}+\int_0^{\frac{1}{3}}\Big)\tau^{s-1}(1-\tau)^{s-1}d\tau\\
    &\qquad+\int_{\frac{5}{3}}^{\infty}\tau^{s-1}((\tau-1)^{s-1}-(1+\tau)^{s-1})d\tau-\int_{\frac{2}{3}}^{\frac{5}{3}}\tau^{s-1}(1+\tau)^{s-1}d\tau \bigg).
\end{align*}
By a change of variables we find that
\begin{align*}
    &\int_{\frac{5}{3}}^{\infty}\tau^{s-1}((\tau-1)^{s-1}-(\tau+1)^{s-1})d\tau \\ 
    &= \lim_{R\rightarrow \infty}\bigg[\int_{\frac{2}{3}}^{R-1} (\tau+1)^{s-1}\tau^{s-1}d\tau-\int_{\frac{5}{3}}^R \tau^{s-1}(\tau+1)^{s-1}d\tau \bigg] \\
    &= \lim_{R\rightarrow \infty}\bigg[\int_{\frac{2}{3}}^{\frac{5}{3}} (\tau+1)^{s-1}\tau^{s-1}d\tau-\int_{R-1}^R \tau^{s-1}(\tau+1)^{s-1}d\tau \bigg]=\int_{\frac{2}{3}}^{\frac{5}{3}} (\tau+1)^{s-1}\tau^{s-1}d\tau,
\end{align*}
since the $\int_{R-1}^R$-integral vanishes. 
Consequently, by Proposition \ref{appendix_prop_1},
\begin{align*}
    u_*'=\frac{1}{2D}\lim_{x\to 0}\lim_{h\to 0}\frac{I_1+I_2+I_3+I_4+I_5}{2h} = \frac{C}{2}\frac{\Gamma(s)^2}{\Gamma(2s)}s=\frac{C}{2}\frac{\Gamma(s)^2}{\Gamma(2s)}(-1)\frac{\Gamma(1-s)}{\Gamma(-s)}.
\end{align*}
\end{proof}

\appendix
\section{}\label{appendix_u_lim}
We now prove Main Theorem \ref{thm::u^n_limit} in the case $n=0$. First, however, we establish some useful identities. Let $\delta_x^2f=f(\cdot +x)-2f(\cdot)+f(\cdot-x)$ denote second order differences. We also introduce 
\begin{align*}
    \Phi_s(\tau) := \delta_1^2(|\tau|^{s-1}),\quad P_s(t):=\int_0^t\Phi_s(\tau)\,d\tau, \quad Q_s(t):=\int_0^t\Phi_s(\tau)\tau^s\,d\tau.
\end{align*}
We write $P_s(\infty)$ and $Q_s(\infty)$ for the corresponding limits as $t\to\infty$.  By \ref{assump_K},
    \begin{align}\label{eq:delta2_K_identity}
        \delta_{x}^2K(\tau x)=\delta_{x}^2K_s(\tau x)+\delta_{x}^2K_{\text{reg}}(\tau x) = \frac{C}{x^{1-s}}\Phi_s(\tau)+\delta_{x}^2K_{\text{reg}}(\tau x), 
    \end{align}
    and the mean value theorem for integrals yields
    \begin{align}\label{eq:K_reg_second_difference}
        \delta_{x}^2K_{\text{reg}}(y) = \int_0^x\int_0^xK''_{\text{reg}}(y+t_1-t_2)dt_1 dt_2 = x^2 K''_{\text{reg}}(y+\xi_y),
    \end{align}
    for a $\xi_y\in(-x,x)$ depending on $y$. The proof of Main Theorem \ref{thm::u^n_limit} in the case $n=0$ relies heavily on the properties of $\Phi_s(\tau)$, so we present the following Lemma:

    \begin{lem}\label{lem:second_diff_properties}
Let $s\in(0,1)$. Then $\Phi_s$ is increasing on $(0,1)$, has a
unique root
    $\tau_s^*\in\left(\frac12,0.72\right),$ and satisfies
\[
    \Phi_s(\tau)<0
    \quad\text{for }0<\tau<\tau_s^*,
    \qquad
    \Phi_s(\tau)>0
    \quad\text{for }\tau>\tau_s^*.
\]
Moreover,
\begin{align}
    \label{eq:integral_identities_appendix_lemma}
        P_s(\infty)=0,
    \qquad
    Q_s(\infty)=\frac{B(s,s)}2=:\beta_s,
\end{align}
and for $0\leq t<1$,
\begin{align*}
    P_s(t)
    =
    \frac{(1+t)^s-(1-t)^s-2t^s}{s}, \qquad
    Q_s(t)
    =    -\frac{t^{2s}}{s}
    +\frac{2t^{s+1}}{s+1}
    +S_{s,\infty}(t),
\end{align*}
where
\begin{align*}
    S_{s,N}(t):=2\sum_{k=1}^{N}
    \frac{\Gamma(s)}
    {\Gamma(2k+1)\Gamma(s-2k)}
    \frac{t^{2k+s+1}}{2k+s+1}\geq 0, \qquad \text{for} \ N\in\N \cup \{\infty\}.
\end{align*}
In particular,
\begin{align}\label{eq:b_s_explicit}
    b_s
    &:=
    -Q_s(\tau_s^*) =
    \frac{(\tau_s^*)^{2s}}{s}
    -\frac{2(\tau_s^*)^{s+1}}{s+1}
    -S_{s,\infty}(\tau^*_s)
    >0.
\end{align}
\end{lem}
    \begin{proof}
        Direct computations as those in \cite[Lemma 5.1]{ehrnstrom_2023_onThePreciseCuspedBehaviour} show the claims concerning the monotonicity, root, and sign of $\Phi_s$, and \eqref{eq:integral_identities_appendix_lemma}. To see that $\tau_s^* \in (\frac12, 0.72)$, note
        \begin{align*}
        \frac{d^2}{ds^2} \big(\Phi_s(a) \big) &= \ln(1+a)^2(1+a)^{s-1}-2\ln(a)^2a^{s-1}+\ln(1-a)^2(1-a)^{s-1}. 
    \end{align*}
    As $a^{s-1}$ and $(1+a)^{s-1}$ with $a\in(0,1)$ are decreasing and increasing in $s$ respectively, 
    \begin{align*}
&        \frac{d^2}{ds^2} \big(\Phi_s(0.5) \big) =\frac{\ln\left(1.5\right)^{2}}{1.5^{1-s}}  -\frac{\ln\left(0.5\right)^{2}}{0.5^{1-s}}
\leq \ln(1.5)^2 -\ln(0.5)^2< -0.3< 0,\\
&\frac{d^2}{ds^2} \big(\Phi_s(0.72) \big) \geq \ln(0.28)^2-\frac{2\ln(0.72)^2}{0.72}+\frac{\ln(1.72)^2}{1.72} > 1.4>0,
    \end{align*}
for all $s\in[0,1]$. Consequently, 
\begin{align*}
    &\frac{d}{ds} \big(\Phi_s(0.5)\big) \geq \frac{d}{ds} \big(\Phi_s(0.5) \big) \bigg\vert_{s=1} > 1.0 >0\\
    &\frac{d}{ds} \big(\Phi_s(0.72)\big) \leq \frac{d}{ds} \big(\Phi_s(0.72) \big) \bigg\vert_{s=1} < -0.05 <0,
\end{align*}
for all $s\in [0,1]$. This then implies that 
\begin{align*}
      &\Phi_s(0.5)<\Phi_1(0.5)=0 \quad \text{and} \quad \Phi_s(0.72)>\Phi_1(0.72)=0, 
    \end{align*}
establishing the result.

The expression for $P_s$ follows by direct integration. For $\tau\in[0,1)$, we have that
\[
    (1+\tau)^{s-1}+(1-\tau)^{s-1}
    =
    2\sum_{k=0}^{\infty}
    \frac{\Gamma(s)}
    {\Gamma(2k+1)\Gamma(s-2k)}
    \tau^{2k}
\]
by the binomial theorem. Consequently,
\begin{align*}
    Q_s(t)
    &=
    \int_0^t
    \left(
        (1+\tau)^{s-1}
        +(1-\tau)^{s-1}
        -2\tau^{s-1}
    \right)\tau^s\,d\tau =
    -\frac{t^{2s}}{s}
    +\frac{2t^{s+1}}{s+1}
    +S_{s,\infty}(t).
\end{align*} Moreover, $S_{s,N}(t)\geq0$ for all $s\in(0,1)$, $t\in[0,1)$, $N\in \N \cup \{\infty \}$ since
\[
    \frac{\Gamma(s)}
    {\Gamma(2k+1)\Gamma(s-2k)}>0,
    \qquad
    s\in(0,1),\quad k\geq1.
\]
Finally, since $\Phi_s<0$ on $(0,\tau_s^*)$ we have
$    b_s=-Q_s(\tau_s^*)>0$, and \eqref{eq:b_s_explicit} follows directly from the expression for $Q_s$.
\end{proof}
The proof of the zeroth-order asymptotics follows \cite{ehrnstrom_2023_onThePreciseCuspedBehaviour},
except for the final sign condition. Establishing this condition
uniformly for $s\in[0.35,1)$ requires sharper estimates and rigorous
interval arithmetic. We isolate this step in the following lemma,
then present the remaining argument for completeness.
\begin{lem}[Sign condition]\label{lem:appendix_sign_condition}
For every $s\in[0.35,1)$, define
    \begin{align}\label{def:f_s_sigma_def}
        f_s(\sigma):= \int_0^{\tau_s^*}\delta_1^2(|\tau|^{s-1})\min(1, \sigma\tau^s) d\tau+\sigma^{-2}b_s+(\tfrac{B(s,s)}{2}+b_s)(1-\sigma^{-1}),
    \end{align}
    for $\sigma \geq 1,$ with $b_s>0$ as defined in Lemma \ref{lem:second_diff_properties}. Then $f_s$ satisfies
    $$f_s(\sigma)>0 \qquad \text{for all } \sigma\in(1,\infty).$$
\end{lem}
The idea is to rewrite the required sign condition in a form suitable for rigorous interval evaluation, and then subdivide the relevant parameter domains and verify positivity on each resulting box. Recall that interval arithmetic replaces each real-valued operation by an interval extension guaranteed to contain all possible values on the input box. Thus, if the interval evaluation of $F$ on $I\times T$ yields $[L,U]$, then
\[
    F(s,t)\in[L,U]
    \qquad\text{for every }(s,t)\in I\times T,
\]
and in particular $\inf_{I\times T}F\geq L$ and $\sup_{I\times T}F\leq U$.
\begin{proof}
     $\bullet$ \textbf{Reformulation of the sign condition.} It is convenient to make the change of variables $z:=\sigma^{-1}\in(0,1]$ and define $\widetilde f_s(z):=f_s(z^{-1})$. Then
\begin{align}\label{eq:f_tilde_def}
    \widetilde f_s(z)
    &=
    \int_0^{\tau_s^*}
    \Phi_s(\tau)
    \min\left(1,\tfrac{\tau^s}{z}\right)\,d\tau
    +z^2b_s
    +(\beta_s+b_s)(1-z).
\end{align}
In particular, $\widetilde f_s(1)=f_s(1)=0$, and it remains to prove that     $\widetilde f_s(z)>0$ for  $s\in[0.35,1),$ $ z\in(0,1).$ Set $a_s(z):=\min\{z^{1/s},\tau_s^*\}$. The minimum in \eqref{eq:f_tilde_def} changes branch at
$\tau=z^{1/s}$, which motivates splitting the integral at this point:
\begin{align}
    \widetilde f_s(z)
    &=
    (1-z)\beta_s+(z^2-z+1)b_s
    +P_s(\tau_s^*)-P_s(a_s(z))
    +\frac{Q_s(a_s(z))}{z} \notag\\
    &=
    (1-z)\beta_s+(z^2-z+1)b_s \notag\\
    &\quad
    +\frac{
        (1+\tau_s^*)^s-(1-\tau_s^*)^s-2(\tau_s^*)^s
        -(1+a_s(z))^s+(1-a_s(z))^s+2a_s(z)^s
    }{s} \notag\\
    &\quad
    -\frac{a_s(z)^{2s}}{sz}
    +\frac{2a_s(z)^{s+1}}{(s+1)z}
    +\frac{S_{s,\infty}(a_s(z))}{z}.
    \label{eq:f_tilde_explicit}
\end{align}
Since $S_{s,N}\leq S_{s,\infty}$, we therefore have
\begin{align}
    \widetilde f_s(z)\geq L_{s,N}(z)
    &:=
    (1-z)\beta_s+(z^2-z+1)b_s \notag\\
    &\quad
    +\frac{
        (1+\tau_s^*)^s-(1-\tau_s^*)^s-2(\tau_s^*)^s
        -(1+a_s(z))^s+(1-a_s(z))^s+2a_s(z)^s
    }{s} \notag\\
    &\quad
    -\frac{a_s(z)^{2s}}{sz}
    +\frac{2a_s(z)^{s+1}}{(s+1)z}
    +\frac{S_{s,N}(a_s(z))}{z}.
    \label{eq:f_tilde_lower_bound}
\end{align}
We will also use the equivalent representation
\begin{align}
    \widetilde f_s(z)
    &=
    (1-z)\left(
        \beta_s-b_s\left(z+\frac1z\right)
    \right)
    +\int_{a_s(z)}^{\tau_s^*}
        \Phi_s(\tau)
        \left(1-\frac{\tau^s}{z}\right)\,d\tau,
    \label{eq:f_tilde_sign_form}
\end{align}
where the integral is nonnegative.

$\bullet$ \textbf{Interval enclosures}.
We now describe the interval bounds used in the computer-assisted verification.
For this purpose we work on the closed interval $s\in[0.35,1]$. Let $I=[s_-,s_+]\subset[0.35,1]$. All interval evaluations below will be performed using outward-rounded interval arithmetic.

Since $\Phi_1\equiv0$, we regularize the root condition by defining
\[
    \Psi_s(\tau):=\frac{\Phi_s(\tau)}{1-s},
    \qquad s\in[0.35,1),
\]
and extending continuously to $s=1$ by $\Psi_1(\tau)
    =
    \log\left(\frac{\tau^2}{1-\tau^2}\right).$
Thus $\Psi_s$ has the same sign and root as $\Phi_s$ for $s<1$, while
$\Psi_1$ has the unique root $\tau_1^*=\frac{1}{\sqrt{2}}$. For the interval evaluation we will use the equivalent nonsingular
representation
\begin{align*}
    &\Psi_s(\tau)
    =
    E_s(1+\tau)-2E_s(\tau)+E_s(1-\tau),
    \qquad s\in[0.35,1],\\
    \text{where} \quad 
    &E_s(x):=
    -\log(x)\,
    \operatorname{exprel}\bigl((s-1)\log x\bigr),\quad 
    \operatorname{exprel}(y):=
    \begin{cases}
        \tfrac{e^y-1}{y}, & y\neq0,\\[1mm]
        1, & y=0.
    \end{cases}
\end{align*}
Indeed, for $s<1$, $ E_s(x)=\frac{x^{s-1}-1}{1-s}$, while $E_1(x)=-\log x$. Thus this representation is well-defined and
continuous at $s=1$, and contains no division by $(1-s)$. Note that for an interval $Y=[y_-,y_+]$ containing zero we do \textit{not} evaluate the
quotient in the definition of $\operatorname{exprel}$ by interval
division. Instead,  since $\operatorname{exprel}(y)=\int_0^1 e^{ty}\,dt$ it is strictly increasing and hence
$\operatorname{exprel}(Y)
    \subseteq
    \bigl[\operatorname{exprel}(y_-),
          \operatorname{exprel}(y_+)\bigr],$
with $\operatorname{exprel}(0)=1$. If
$t_-,t_+\in(1/2,0.72)$ satisfy
 $\Psi_s(t_-)<0$ and $    \Psi_s(t_+)>0 $ for $ s\in I$,
then $\tau_s^*\in T_I:=[t_-,t_+]$ for $ s\in I.$ Such bounds are obtained by interval evaluation and subdivision as necessary. 

To obtain a rigorous enclosure of $b_s$, write the $k$-th summand of
$S_{s,\infty}(t)$ as
\[
    q_{s,k}(t)
    :=
    2c_{s,k}
    \frac{t^{2k+s+1}}{2k+s+1}, \qquad c_{s,k}:=
    \frac{\Gamma(s)}
         {\Gamma(2k+1)\Gamma(s-2k)}
\]
Although we retain the Gamma-function notation, for the interval
evaluation we use the equivalent recurrence
\[
    c_{s,1}=\frac{(1-s)(2-s)}{2},\qquad
    c_{s,k+1}
    =
    c_{s,k}
    \frac{(2k+1-s)(2k+2-s)}
         {(2k+1)(2k+2)}.
\]
This also gives the continuous extension $c_{1,k}=0$, avoiding the
Gamma-function poles at $s=1$. For $s<1$, the recurrence relation gives
\[
    \frac{q_{s,k+1}(t)}{q_{s,k}(t)}
    =
    t^2
    \frac{(2k+1-s)(2k+2-s)}
         {(2k+1)(2k+2)}
    \frac{2k+s+1}{2k+s+3}
    <t^2.
\]
At $s=1$ all the terms $q_{s,k}$ vanish, so the resulting tail
estimate remains valid on the closed interval $s\in[0.35,1]$.

Consequently, with $R_{s,M}(t):=S_{s,\infty}(t)-S_{s,M}(t)$,
we have, for $t\leq0.72$,
\begin{align}\label{eq:S_tail_bound}
    0\leq q_{s,M+1}(t)
\leq  R_{s,M}(t)
    \leq \frac{q_{s,M+1}(t)}{1-t^2}
    \leq \frac{q_{s,M+1}(t)}{1-0.72^2}.
\end{align}
Since $\frac{d}{dt}\bigl(-Q_s(t)\bigr)=-\Phi_s(t)t^s$, Lemma \ref{lem:second_diff_properties} implies that $-Q_s$ attains
its maximum at $t=\tau_s^*$, with value $b_s$.
Moreover, the tail bounds \eqref{eq:S_tail_bound} give
\begin{align}\label{eq:B_s_bounds}
\frac{t^{2s}}{s}
-\frac{2t^{s+1}}{s+1}
-S_{s,M}(t)
-\frac{q_{s,M+1}(t)}{1-t^2}
\leq -Q_s(t)
\leq
\frac{t^{2s}}{s}
-\frac{2t^{s+1}}{s+1}
-S_{s,M}(t)
-q_{s,M+1}(t).
\end{align}
Suppose $\tau_s^*\in T_I=[t_-,t_+]$ for every $s\in I$.
Since $b_s\geq -Q_s(t_\pm)$, interval evaluation of the left-hand
side of \eqref{eq:B_s_bounds} on $I\times\{t_-\}$ and
$I\times\{t_+\}$ gives two lower bounds for $b_s$; we take
the larger of their lower endpoints.
Interval evaluation of the right-hand side on the whole box
$I\times T_I$ gives an upper bound, since this box contains
$(s,\tau_s^*)$ for every $s\in I$.
Thus we obtain
\[
b_s\in[\underline b_I,\overline b_I],
\qquad s\in I.
\]
Similarly, $\beta_s$ is evaluated directly by interval arithmetic.
No tail estimate is needed in \eqref{eq:f_tilde_lower_bound},
where we only use $S_{s,N}\leq S_{s,\infty}$.

$\bullet$ \textbf{Splitting the $z$-domain.} We now divide the remaining verification into three regions in $z$, each yielding simplifications that are useful for the interval arithmetic.

\textbf{Region I: $z\in(0,\frac{1}{2}]$.}
Since $s\leq 1$ and $\tau_s^*>\frac12$, we have
$z^{1/s}\leq z\leq \frac12<\tau_s^*,$ and hence $a_s(z)=z^{1/s}.$ From  \eqref{eq:f_tilde_lower_bound} we then get
\begin{align}
    \widetilde f_s(z)\geq L^{(1)}_{s,N}(z)
    &:=
    (1-z)\beta_s+(z^2-z+1)b_s \notag\\
    &\quad
    +\frac{
        (1+\tau_s^*)^s-(1-\tau_s^*)^s-2(\tau_s^*)^s
        -(1+z^{1/s})^s+(1-z^{1/s})^s
    }{s} \notag\\
    &\quad
    +\frac{z}{s}
    +\frac{2z^{1/s}}{s+1}
    +2\sum_{k=1}^N c_{s,k} \frac{z^{(2k+1)/s}}{2k+s+1}.
    \label{eq:f_tilde_lower_bound_region1}
\end{align}
The right-hand side extends continuously to $z=0$, so
\eqref{eq:f_tilde_lower_bound_region1} can be verified by interval
arithmetic on the closed region $(s,z)\in[0.35,1]\times[0,\frac{1}{2}]$.

\textbf{Region II: $z\in[\frac12,0.9]$.}
In this region the switching point $z^{1/s}$ may lie on either side of
$\tau_s^*$, so we retain the general lower bound
\eqref{eq:f_tilde_lower_bound}. Let $I\subset[0.35,1]$ and
$Z\subset[\frac12,0.9]$ be interval boxes, and let
$T_I$ be the corresponding enclosure of $\tau_s^*$ obtained above.
On a box $I\times Z$, we evaluate $z^{1/s}$ and use the enclosure
$T_I$ of $\tau_s^*$. If these enclosures are disjoint, the branch of
$a_s(z)=\min\{z^{1/s},\tau_s^*\}$ is determined throughout the box.
Otherwise, $a_s(z)$ is enclosed directly by interval arithmetic, and
the box is subdivided further if the resulting bound for
$L_{s,N}(z)$ is not sufficiently sharp. On each resulting box we evaluate $L_{s,N}(z)$ using outward-rounded
interval arithmetic. A strictly positive lower endpoint certifies
$\widetilde f_s(z)>0$ on the entire box.

\textbf{Region III: $z\in[0.9,1)$.}
For $s\in[0.35,1]$ and $z\in[0.9,1)$ we have
\[
    z^{1/s}\geq 0.9^{1/0.35}=0.9^{20/7}>0.72>\tau_s^*,
\]
and therefore $a_s(z)=\tau_s^*$. Hence the integral in \eqref{eq:f_tilde_sign_form} vanishes and
\[
    \widetilde f_s(z)
    =
    (1-z)\left(
        \beta_s-b_s\left(z+\tfrac1z\right)
    \right).
\]
Since $z\mapsto z+z^{-1}$ is decreasing on $(0,1]$,  $z+\frac1z
    \leq
    0.9+\frac1{0.9}
    =
    \frac{181}{90}$. Consequently
\begin{align}\label{eq:f_tilde_region3}
    \widetilde f_s(z)
    \geq
    (1-z)\left(
        \beta_s-\tfrac{181}{90}b_s
    \right).
\end{align}
Since $1-z>0$, it remains only to verify
\[
    \beta_s-\tfrac{181}{90}b_s>0,
    \qquad s\in[0.35,1].
\]
This is checked by interval arithmetic on the $s$-intervals $I$, using
the lower interval bound for $\beta_s$ and the upper bound
$\overline b_I$ for $b_s$ obtained above.

$\bullet$ \textbf{Computer verification output.}
The interval computations were performed with Python 3.12.4
and \texttt{python-flint} 0.9.0 (FLINT 3.6.0),
using Arb ball arithmetic at 320-bit precision.
We used $N=2$ in the finite lower bounds and $M=16$ in the auxiliary
series-tail enclosure of $b_s$. All subdivision endpoints were exact rational
numbers, and all numerical bounds were evaluated with outward rounding.

The computation certified $L^{(1)}_{s,2}(z)>0$ on
$[0.35,1]\times[0,\frac{1}{2}]$ and $L_{s,2}(z)>0$ on
$[0.35,1]\times[\frac{1}{2},0.9]$. It also certified
$H_s:=\beta_s-\frac{181}{90}b_s>0$ for every $s\in[0.35,1]$.
The numbers of terminal boxes (intervals for $H_s$) and conservative
uniform lower bounds obtained were as follows:
\begin{center}
\begin{tabular}{lrr}
\hline
Quantity & Certified boxes/intervals & Lower bound \\
\hline
$L^{(1)}_{s,2}(z)$ & 5612 & $0.000001144938$ \\
$L_{s,2}(z)$ & 124459 & $0.000000004352$ \\
$H_s$ & 18 & $0.019678658214$ \\
\hline
\end{tabular}
\end{center}
These lower bounds are the smallest stored certified margins in each region,
rounded further down for display; they are not estimates of the true minima.
No unresolved boxes remained. Exact coverage of the stated domains was checked,
and every recorded bound was re-evaluated successfully from the embedded
machine-readable certificate. The verification code and certificate are
contained in the supplementary notebook
\texttt{new\_verify\_appendixA.ipynb}\footnote{The notebook is also available on the author's
\href{https://github.com/robinol99/whitham_paper_verify_appendix/blob/main/new_verify_appendixA.ipynb}{GitHub page}.}.

Thus the required sign condition
$\widetilde f_s(z)>0$ holds for $s\in[0.35,1]$ and $z\in(0,1)$,
where $s=1$ is interpreted by continuous extension. In particular, the sign
condition holds throughout the range $s\in[0.35,1)$ of the lemma statement.
\end{proof}

\nc

\begin{proof}[Proof of the $n=0$ case of Main Theorem \ref{thm::u^n_limit} \cite{ehrnstrom_2023_onThePreciseCuspedBehaviour}]
    Let $$g(x):=\frac{u(x)}{x^s}, \qquad m:= \liminf_{x\to 0} g(x), \qquad M:= \limsup_{x\to 0} g(x).$$
    We will show $m=M=\frac{C}{2}B(s,s)=\frac{C}{2}\frac{\Gamma(s)^2}{\Gamma(2s)}$, where $B$ denotes the beta function.
    
    \medskip
    \noindent \textbf{Step 1: $m>0$ and $M<\infty$.}    
    From \eqref{eq:u_second_difference} we have that 
    \begin{align}\label{eq:u_lim_expression}
        g(x)^2=\frac{1}{x^{2s}}\Big[\int_0^x+\int_x^\nu +\int_\nu^\infty \Big] \delta_x^2K(y)u(y) dy=:J_1+J_2+J_3.
    \end{align}
    Since $u$ is monotone on $[0,\nu]$ (assumption \ref{assump_Ulim}) we can apply the second mean value theorem for integrals for $J_1$:
    \begin{align*}
        J_1 = \frac{u(x)}{x^{2s}}\int_{\lambda(x)x}^x \delta_x^2K(y)dy = g(x)\int_{\lambda(x)}^1\Big( C\delta_1^2(|\tau|^{s-1})+ x^{3-s} K''_{\text{reg}}(\tau x +\xi_{\tau x}) \Big) d\tau,
    \end{align*}
    for some $\lambda(x)\in(0,1)$. Here we have also used \eqref{eq:delta2_K_identity} and \eqref{eq:K_reg_second_difference}. For $J_2$ we have by \eqref{eq:delta2_K_identity}, the definition of $g$, and \eqref{eq:K_reg_second_difference} that
    \begin{align*}
        J_2=C\int_1^{\frac{\nu}{x}}\delta_1^2(|\tau|^{s-1})\tau^s g(\tau x) d\tau+x^{2-2s}\int_x^\nu K''_{\text{reg}}(y+\xi_y)u(y)dy,
    \end{align*}
    and for $J_3$ we note that
    \begin{align*}
        0\leq \int_\nu^\infty  \delta_{x}^2K(y)u(y) dy \leq -\| u\|_\infty K'(\nu-x)x^2
    \end{align*}
    The lower bound follows  by the convexity of $K$ (assumption \ref{assump_Ulim}) which yields $\delta_{x}^2K(y)\geq 0$ for $0\leq x <y$, while the upper follows by the mean-value theorem and $-K'$ being non-increasing. Consequently, with 
    $$\underline{g}(x):=\min_{y\in [x,\nu]}g(y), \qquad \overline{g}(x):=\max_{y\in [x,\nu]}g(y),$$
    we have the estimates 
    \begin{align}
        &g(x)^2 \gtrsim g(x)\Big(\int_{0}^1 \delta_1^2(|\tau|^{s-1}) d\tau +O(x^{3-s})\Big)+\underline{g}(x)\Big(\int_1^{\frac{\nu}{x}}\delta_1^2(|\tau|^{s-1})\tau^s d\tau +O(x^{2-2s})\Big),\label{eq:g_grt} \\
        &g(x)^2 \lesssim g(x)\Big(\int_{\tau_s^*}^1 \delta_1^2(|\tau|^{s-1}) d\tau +O(x^{3-s})\Big)+\overline{g}(x)\Big(\int_1^{\frac{\nu}{x}}\delta_1^2(|\tau|^{s-1})\tau^s d\tau +O(x^{2-2s})\Big), \label{eq:g_leq}
    \end{align}
    as $x\to 0$. The first integral bounds in each line follow from Lemma \ref{lem:second_diff_properties}.

    Assume now $m=0$ for the sake of contradiction. Then we may pick a realizing sequence $\{ x_k\}_{k\in \N} \subset (0,\nu]$ for $m$ in such a way that $g = \underline{g}$ along the sequence. Evaluating \eqref{eq:g_grt} at this $x_k$ we can divide by $g(x_k)$ and let $k\to \infty$, yielding
    \begin{align*}
        m \gtrsim \int_0^1 \delta_1^2(|\tau|^{s-1}) d\tau +\int_1^{\infty}\delta_1^2(|\tau|^{s-1}) \tau^s d\tau = \int_1^{\infty}\delta_1^2(|\tau|^{s-1}) (\tau^s-1) d\tau >0,
    \end{align*}
    by the positivity of the integrand. The equality follows from Lemma \ref{lem:second_diff_properties}. Thus we have a contradiction, so $m>0$. Similar arguments for \eqref{eq:g_leq} yield $M<\infty$.

    \medskip
    \noindent \textbf{Step 2: Sharper bounds for $m$ and $M$.} Now that we know $0<m\leq M <\infty$ we can derive the sharper bounds
    \begin{align}
        M^2 &\leq C\Big(m \int_0^{\tau_s^*}\delta_1^2(|\tau|^{s-1})\tau^s d\tau + M \int_{\tau_s^*}^\infty \delta_1^2(|\tau|^{s-1}) \tau^s d\tau\Big), \label{eq:M_sharper_bound} \\
                m^2 &\geq C\Big(\int_0^{\tau_s^*}\delta_1^2(|\tau|^{s-1})\min(m, M\tau^s) d\tau + m \int_{\tau_s^*}^\infty \delta_1^2(|\tau|^{s-1}) \tau^s d\tau\Big), \label{eq:m_sharper_bound}
    \end{align}
    as follows: Knowing $g$ is bounded lets us use \eqref{eq:delta2_K_identity}, $J_2$ and $J_3$ to replace \eqref{eq:u_lim_expression} by 
    \begin{align*}
        g(x)^2=C\int_0^{\frac{\nu}{x}} \delta_1^2(|\tau|^{s-1})\tau^s g(\tau x) d\tau  + O(x^{2-2s})
    \end{align*}
    as $x\rightarrow 0$. Then, since $\delta_1^2(|\tau|^{s-1})$ is negative on $(0,\tau_s^*)$ and positive on $(\tau_s^*,\infty)$, 
    \begin{align*}
        M^2\leq C(\inf_{y \in (0,\nu]} g(y)) \int_0^{\tau_s^*}  \delta_1^2(|\tau|^{s-1})\tau^s d\tau + C(\sup_{y \in (0,\nu]}g(y)) \int_{\tau_s^*}^\infty \delta_1^2(|\tau|^{s-1})\tau^s d\tau,
    \end{align*}
    whence letting $\nu \rightarrow 0$ yields \eqref{eq:M_sharper_bound}. For \eqref{eq:m_sharper_bound}, note that 
    $$\tau^s g(\tau x) \leq \min \big(g(x), \tau^s\sup_{y\in(0,\nu]} g(y)\big),$$
    for every $\tau \in (0,1)$, $x\in (0,\nu]$, since $u$ is increasing on $(0,\nu]$. 
    This yields 
    \begin{align*}
        g(x)^2 &\geq C\int_0^{\tau_s^*} \delta_1^2(|\tau|^{s-1}) \min \big(g(x), \tau^s\sup_{y\in(0,\nu]} g(y)\big) d\tau\\
        &\qquad +C(\inf_{y\in(0,\nu]} g(y)) \int_{\tau_s^*}^\infty  \delta_1^2(|\tau|^{s-1})\tau^s d\tau + O(x^{2-2s}) \qquad \text{ as $x\to 0$}.
    \end{align*}
   Taking the limit along a sequence realizing $m$, then letting $\nu\to 0$, yields \eqref{eq:m_sharper_bound}.

    \medskip 
    \noindent \textbf{Step 3: $m=M$.} We first define $\sigma:=\frac{M}{m} \geq 1$ and rewrite \eqref{eq:M_sharper_bound} and \eqref{eq:m_sharper_bound} purely in terms of $\sigma$ and $m$:
    \begin{align}
        m &\leq C\Big(\sigma^{-2} \int_0^{\tau_s^*}\delta_1^2(|\tau|^{s-1})\tau^s d\tau + \sigma^{-1} \int_{\tau_s^*}^\infty \delta_1^2(|\tau|^{s-1}) \tau^s d\tau\Big), \label{eq:M_sharper_bound_sigma} \\
                m &\geq C\Big(\int_0^{\tau_s^*}\delta_1^2(|\tau|^{s-1})\min(1, \sigma\tau^s) d\tau + \int_{\tau_s^*}^\infty \delta_1^2(|\tau|^{s-1}) \tau^s d\tau\Big). \label{eq:m_sharper_bound_sigma}
    \end{align}
    When $\sigma=1$, both right-hand sides evaluate to $C\int_0^\infty \delta_1^2(|\tau|^{s-1})\tau^s d\tau = C\frac{B(s,s)}{2}$, so $C\frac{B(s,s)}{2}=m=\frac{M}{\sigma}=M$. Thus, if no $m>0$ satisfies both \eqref{eq:m_sharper_bound_sigma} and \eqref{eq:M_sharper_bound_sigma} when $\sigma>1$ we are done. To show this, consider $f_s(\sigma)$ defined in \eqref{def:f_s_sigma_def}. Note that $Cf_s(\sigma)$ is the right-hand side of \eqref{eq:m_sharper_bound_sigma} minus that of \eqref{eq:M_sharper_bound_sigma}. Thus, we have that $f_s(1)=0$ and want to show that $f_s$ is positive on $(1,\infty)$.  The positivity of $f_s(\sigma)$ on $(1,\infty)$ was established in Lemma \ref{lem:appendix_sign_condition},  concluding the proof.\nc
\end{proof}

\section{}\label{appendix_hypergeom}
Computing the final expression in \eqref{eq:big_computation} requires some tedious calculations. The first step is to translate the expression to the language of hypergeometric functions so we can borrow from its extensive literature. 
For the sake of being self-contained we recap the necessary definitions and results from the theory. We follow the notation of \cite{duverney2024introduction}, and the formulas below are taken from \cite{duverney2024introduction} and \cite{abramowitz_stegun}.

We define, for fixed $a,b \in \R$, $c \in \R\setminus \Z_{\leq 0}$, the \textit{Gauss hypergeometric function}
\begin{align}\label{eq:def_gauss_hypergeom}
{}_2F_1\left( \begin{array}{c} a,\, b \\ c \end{array} \middle| \  x \right) = \frac{\Gamma(c)}{\Gamma(a)\Gamma(b)}\sum_{k=0}^{\infty} \frac{\Gamma(a+k)\Gamma(b+k)}{\Gamma(c+k)} \frac{x^k}{k!}, \qquad  |x|<1, 
\end{align}
which can (by the integral representation below) be uniquely extended by analytic continuation on $(-\infty, 1)$. 
Note that ${}_2F_1$ is symmetric in $a$ and $b$. When $a=c$ or $b=c$, ${}_2F_1$ reduces to the binomial function \cite[15.1.8]{abramowitz_stegun}:
\begin{align}\label{eq::gauss_hypergeom_binomial}
    {}_2F_1\left( \begin{array}{c} a, b \\ b \end{array} \middle| x \right) =(1-x)^{-a}, \qquad {}_2F_1\left( \begin{array}{c} a, b \\ a \end{array} \middle| x \right) =(1-x)^{-b}.
\end{align}
Especially important for us is the following integral representation \cite[(4.10)]{duverney2024introduction}:
\begin{align}\label{eq::hypergeom_integral_repr}
{}_2F_1\left( \begin{array}{c} a, b \\ c \end{array} \middle| x \right)
&= \frac{\Gamma(c)}{\Gamma(b)\Gamma(c-b)} \int_0^1 t^{b-1} (1-t)^{c-b-1} (1-xt)^{-a} \, dt,
\end{align}
valid for $c>b>0$, $x\in(-\infty,1)$. We will need Euler's transformation formula \cite[Theorem 4.8]{duverney2024introduction}:
\begin{align}\label{eq:euler_transf_form}
    {}_2F_1\left( \begin{array}{c} a,\, b \\ c \end{array} \middle| \  x \right) = (1-x)^{-a}{}_2F_1\left( \begin{array}{c} a,\, c-b \\ c \end{array} \middle| \  \frac{x}{x-1} \right), \qquad x \in (-\infty, 1).
\end{align}
We also need the following linear transformation formula  \cite[(15.3.7)]{abramowitz_stegun}:

\begin{align}
&\begin{aligned}
   {}_2F_1\left( \begin{array}{c} a,\, b \\ c \end{array} \middle| x \right)
   &= (- x)^{-a} \frac{\Gamma(c)\Gamma(b - a)}{\Gamma(b)\Gamma(c - a)}
\, {}_2F_1\left( \begin{array}{c} a,\, 1-c +a \\ 1-b+a \end{array} \middle| \ \frac{1}{x} \right) \\
&\qquad  + (- x)^{-b} \frac{\Gamma(c)\Gamma(a - b)}{\Gamma(a)\Gamma(c - b)}
\, {}_2F_1\left( \begin{array}{c} b,\, 1-c+b \\ 1-a+b \end{array} \middle| \ \frac{1}{x} \right),
\end{aligned}\label{eq:linear_transf_form_2}
\end{align}
which is valid for $x\in(-\infty, 0)$ and  $a-b\notin\Z$. Below we will, without mention, repeatedly use the identity $z\Gamma(z)=\Gamma(z+1)$.

\begin{prop}\label{appendix_prop_1} Let $s\in(0,1)$, $n\in\N\cup\{0 \}$ with the convention $\sum_{i=0}^{-1}=0$. Then
    \begin{align*}
        &\left(\int_{0}^{\frac{2}{3}} +\int_{0}^{\frac{1}{3}}\right)\tau^{s-1}(1-\tau)^{s-1-n} d\tau -3^{n+1-2s}\sum_{i=0}^{n-1} \frac{1}{2^{i+1-s}}\frac{\Gamma(i+1-s)\Gamma(n-i-s)}{\Gamma(1-s)\Gamma(n+1-s)} \\
        &\qquad= \begin{cases}
            \frac{\Gamma(s)^2}{\Gamma(2s)}, \qquad &n=0,\\
            0, \qquad &n\geq 1, \quad s=\frac{1}{2},\\
            \frac{\Gamma(s)\Gamma(s-n)}{\Gamma(2s-n)}, \qquad & n\geq 1,  \quad s\neq \frac{1}{2}.
        \end{cases}
    \end{align*}
\end{prop}
\begin{proof}
    The $n=0$ claim follows directly by the change of variables $\tau \mapsto 1-\tau$ and \eqref{eq::hypergeom_integral_repr}, so consider now $n\geq 1$. By changing variables and \eqref{eq::hypergeom_integral_repr} we have that 
    \begin{align}
        &\int_{0}^{\frac{2}{3}} \frac{ \tau^{s-1} d\tau}{(1-\tau)^{n+1-s}}  = \int_{0}^{1} \frac{\left(\frac{2}{3} \right)^s \tau^{s-1} d\tau}{(1-\tfrac{2}{3}\tau)^{n+1-s}} =  \frac{\left(\tfrac{2}{3} \right)^s}{s} {}_2F_1\left( \begin{array}{c} n+1-s,\, s \\[1mm] s+1 \end{array} \middle| \frac{2}{3} \right),\label{eq:prop1_first_term} \\
        &\int_{0}^{\frac{1}{3}} \frac{ \tau^{s-1} d\tau}{(1-\tau)^{n+1-s}} = \int_{0}^{1} \frac{\left(\frac{1}{3} \right)^s \tau^{s-1}  d\tau}{(1-\tfrac{1}{3}\tau)^{n+1-s}}  =  \frac{\left(\tfrac{1}{3} \right)^s}{s} {}_2F_1\left( \begin{array}{c} n+1-s,\, s \\[1mm] s+1 \end{array} \middle| \frac{1}{3} \right).\label{eq:prop1_second_term}
    \end{align}
    Furthermore, 
    \begin{equation}\label{eq:computation_1}
    \begin{aligned}
        &3^{n+1-2s}\sum_{i=0}^{n-1} \frac{1}{2^{i+1-s}}\frac{\Gamma(i+1-s)\Gamma(n-i-s)}{\Gamma(1-s)\Gamma(n+1-s)} \\
        &=3^{n+1-2s}\Bigg[\frac{2^{s-1}}{n-s} {}_2F_1\left( \begin{array}{c} 1-s,\, 1 \\[1mm] 1+s-n \end{array} \middle| -\frac{1}{2} \right) \\
        &\hspace{5cm}+\frac{2^{s-1-n}}{s} {}_2F_1\left( \begin{array}{c} n+1-s,\, 1 \\[1mm] s+1 \end{array} \middle| -\frac{1}{2} \right)\Bigg].
    \end{aligned}
        \end{equation}
    To see this we use \eqref{eq:def_gauss_hypergeom} and Euler's reflection formula to write
    \begin{equation}\label{eq:comp_2}
         \begin{aligned}
        &{}_2F_1\left( \begin{array}{c} 1-s,\, 1 \\[1mm] 1+s-n \end{array} \middle| -\frac{1}{2} \right) = \frac{\Gamma(1+s-n)}{\Gamma(1-s)}\sum_{j=0}^\infty \frac{\Gamma(j+1-s)}{\Gamma(1+s-n+j)}\left(-\frac{1}{2} \right)^j \\
        &=\frac{\Gamma(n+1-s)}{\Gamma(n+1-s)}\frac{\Gamma(1+s-n)}{\Gamma(1-s)}\sum_{j=0}^\infty \frac{\Gamma(j+1-s)}{\Gamma(1+s-n+j)}\frac{\Gamma(n-j-s)}{\Gamma(n-j-s)} \left(-\frac{1}{2} \right)^j \\
        &=\frac{(n-s)2^{1-s}}{\Gamma(n+1-s)\Gamma(1-s)}\sum_{j=0}^\infty \frac{\Gamma(j+1-s)\Gamma(n-j-s)}{2^{j+1-s}},
    \end{aligned}
    \end{equation}
    and similarly, but now also with the change of variables $j\mapsto (j+n)$, 
    \begin{equation}\label{eq:comp_3}
            \begin{aligned}
        &{}_2F_1\left( \begin{array}{c} n+1-s,\, 1 \\[1mm] s+1 \end{array} \middle| -\frac{1}{2} \right) = \frac{\Gamma(1+s)}{\Gamma(n+1-s)}\sum_{j=0}^\infty \frac{\Gamma(n+1-s+j)}{\Gamma(1+s+j)}\left(-\frac{1}{2} \right)^j \\
        &= \frac{\Gamma(1-s)}{\Gamma(1-s)}\frac{\Gamma(1+s)}{\Gamma(n+1-s)}\sum_{j=n}^\infty \frac{\Gamma(j+1-s)}{\Gamma(1+s-n+j)}\frac{\Gamma(n-j-s)}{\Gamma(n-j-s)}\left(-\frac{1}{2} \right)^{j-n} \\
        &= -\frac{2^{n+1-s}s}{\Gamma(n+1-s)\Gamma(1-s)}\sum_{j=n}^\infty \frac{\Gamma(j+1-s)\Gamma(n-j-s)}{2^{j+1-s}}.
    \end{aligned}
    \end{equation}
    By Euler's transformation \eqref{eq:euler_transf_form},
    \begin{align}\label{eq:comp_4}
         \frac{\left(\tfrac{1}{3} \right)^s}{s} {}_2F_1\left( \begin{array}{c} n+1-s,\, s \\[1mm] s+1 \end{array} \middle| \frac{1}{3} \right) = \frac{(\tfrac{3}{2})^{n+1-s}}{3^ss} {}_2F_1\left( \begin{array}{c} n+1-s,\, 1 \\[1mm] s+1 \end{array} \middle| -\frac{1}{2} \right),
    \end{align}
    so the final term in \eqref{eq:computation_1} cancels with \eqref{eq:prop1_second_term} when $n\geq 1$. 

    For \eqref{eq:prop1_first_term}, Euler's transformation \eqref{eq:euler_transf_form}, followed by \eqref{eq:linear_transf_form_2} and simplifying, yields\footnote{The application of \eqref{eq:linear_transf_form_2} here is still valid when $s=\frac12$, where we simply get $\frac{1}{\Gamma(2\cdot\frac{1}{2}-n)}=0$. For the skeptical reader, we note that one can avoid this pole entirely by treating the case $s=\tfrac 12$ separately, proving the result by induction using recurrence relations \cite[(4.21) and (4.22)]{duverney2024introduction}.}
    \begin{align*}
        &\left(\frac{2}{3} \right)^s \frac{1}{s} {}_2F_1\left( \begin{array}{c} n+1-s,\, s \\[1mm] s+1 \end{array} \middle| \frac{2}{3} \right) =  \frac{3^{n+1-2s}2^s}{s} {}_2F_1\left( \begin{array}{c} n+1-s,\, 1 \\[1mm] s+1 \end{array} \middle| -2\right) \\
        &=(\tfrac 32)^{n+1-2s}\frac{\Gamma(s)\Gamma(s-n)}{\Gamma(2s-n)}{}_2F_1\left( \begin{array}{c} n+1-s,\, n+1-2s \\[1mm] n+1-s \end{array} \middle| -\frac{1}{2}\right) \\
        &\qquad + \frac{3^{n+1-2s}2^{s-1}}{n-s}{}_2F_1\left( \begin{array}{c} 1-s,\, 1 \\[1mm] 1+s-n \end{array} \middle| -\frac{1}{2}\right) \\
        &=\frac{\Gamma(s)\Gamma(s-n)}{\Gamma(2s-n)}+\frac{3^{n+1-2s}2^{s-1}}{n-s}{}_2F_1\left( \begin{array}{c} 1-s,\, 1 \\[1mm] 1+s-n \end{array} \middle| -\frac{1}{2}\right)
    \end{align*}
    In the second equality we have also used that ${}_2F_1$ is symmetric in $a$ and $b$. The final equality follows from \eqref{eq::gauss_hypergeom_binomial}. We see that the final term cancels the first term in \eqref{eq:computation_1} when $n\geq 1$, and when $n=0$ it cancels with \eqref{eq:comp_4}. 
\end{proof}

\section*{Competing interests}
The author declares that they have no competing interests or other interests that might be perceived to influence the results and/or discussion reported in this paper.

\section*{Acknowledgments}
 R. \O. Lien received funding from the Research Council of Norway under Grant Agreement No. 325114 “IMod. Partial differential equations, statistics and data: An interdisciplinary approach to data-based modelling”. The author would like to thank Professor Mats Ehrnström for generous guidance, many helpful discussions, and detailed feedback throughout this project.
The author also thanks ChatGPT versions 5.2--6 for assistance with brainstorming, exploratory work and general feedback, in particular for helping to test approaches, organize the estimates and implement the code in Appendix~\ref{appendix_u_lim}. 

\bibliographystyle{plain} 
\bibliography{references}

\end{document}